\documentclass[11pt]{article}
\usepackage{amsmath, amssymb, amsfonts, amsxtra, latexsym, amscd, eucal, amsthm, graphicx, epsfig, color, enumerate}      
\usepackage{hyperref}   
\usepackage{subcaption}
\newtheorem{Def}{Definition}[section]

\newtheorem{Prop}[Def]{Proposition}
\newtheorem{Lem}[Def]{Lemma}
\newtheorem{Thm}[Def]{Theorem}

\theoremstyle{definition}
\newtheorem{Rem}[Def]{Remark}

\newcommand{\e}{\mathbb{E}}
\newcommand{\bR}{\mathbb{R}}

\newcommand{\mf}{\mathcal{F}}

\newcommand{\pr}{\mathbb{P}}

\newcommand{\bE}{\mathbb{E}}
 
\newcommand{\br}{\mathbb{R}} 

\newcommand{\sm}{\sigma}

\newcommand{\Om}{\Omega}

\newcommand{\sd}{\underline{s}}

\newcommand{\Xm}{\widehat{X}}

\usepackage[top=2.5cm, bottom=2.5cm, left=2cm, right=2cm]{geometry}

\allowdisplaybreaks

\begin{document}
\title{A Multi-level Monte Carlo simulation for invariant distribution of  Markovian switching L\'evy-driven SDEs with super-linearly growth coefficients}

\author{ Hoang-Viet Nguyen\footnote{Hanoi National University of Education. Email: vietnh@hnue.edu.vn} \quad Trung-Thuy Kieu\footnote{Hanoi National University of Education. Email: thuykt@hnue.edu.vn} \quad  Duc-Trong Luong\footnote{Corresponding author. Hanoi National University of Education. Email: trongld@hnue.edu.vn}\\  Hoang-Long Ngo\footnote{Hanoi National University of Education. Email: ngolong@hnue.edu.vn}  \quad Ngoc Khue Tran\footnote{Faculty of Mathematics and Informatics, Hanoi University of Science and Technology, 1 Dai Co Viet, Hai Ba Trung, Hanoi, Vietnam. Email: khue.tranngoc@hust.edu.vn}}

\maketitle

\begin{abstract}

 This paper concerns the numerical approximation for the invariant distribution of Markovian switching L\'evy-driven stochastic differential equations.
 By combining the tamed-adaptive Euler-Maruyama scheme with the Multi-level Monte Carlo method, we propose an approximation scheme that can be applied to stochastic differential equations with super-linear growth drift and diffusion coefficients.  
 \end{abstract} 


\section{Introduction}

On a complete filtered probability space $(\Om, \mf, (\mf_t)_{t\geq 0}, \pr)$, we consider the $d$-dimensional process $X=(X_t)_{t \geq 0}=(X_{1,t},X_{2,t},\ldots,X_{d,t})_{t\geq 0}$ solution to the following L\'evy-driven stochastic differential equation with Markovian switching  
\begin{equation} \label{eqn1}
X_t=x_0+\int_0^t b(\theta_s,X_s)ds+\int_0^t \sigma(\theta_s,X_s)dW_s + \int_0^t \gamma \left(\theta_{s-},X_{s-}\right)dZ_s, 
\end{equation}
where $x_0\in \mathbb{R}^d$; $W=(W_t)_{t\geq 0}=(W_{1,t},W_{2,t},\ldots,W_{d,t})_{t\geq 0}$ is a $d$-dimensional standard Brownian motion; $Z=(Z_t)_{t\geq 0}=(Z_{1,t},Z_{2,t},\ldots,Z_{d,t})_{t\geq 0}$ is a $d$-dimensional centered pure jump L\'evy process whose L\'evy measure $\nu$ satisfies $\int_{\mathbb{R}^d}(1\wedge \vert z\vert^2)\nu(dz)<+\infty$; 
 $\theta = (\theta_t)_{t \ge 0}$ is a Markov chain taking values in a finite state space $S = \{1,\ldots, N\}$. We suppose that  the generator $\Theta = (\vartheta_{ij})_{N\times N}$ of $\theta$ is defined by
$$
\mathbb{P}(\theta_{t+u} = j|\theta_t=i) = \begin{cases}
\vartheta_{ij}u + o(u) & \textnormal{if }i\neq j \\
1+\vartheta_{ii}u + o(u) & \textnormal{if }i=j
\end{cases},
$$
where $t, u > 0$, $\vartheta_{ij} \ge 0$ is the switching rate from state $i$ to state $j \neq i$, and $\vartheta_{ii} = - \sum_{j\in S, j \not = i}\vartheta_{ij}.$ 
Let $N(dt, dz)=(N_1(dt, dz), \ldots, N_d(dt,dz))$ and $\widetilde{N}(dt,dz)$ be the Poisson random measure and compensated Poisson random measure associated to $Z$, respectively. We suppose that $N_1, \ldots, N_d$ are independent and let $\nu_i(dz_i)dt$ denote the intensity measure of $N_i$.  
The  L\'evy-It\^o decomposition of $Z$ takes the form 
$$
Z_{i,t}=\int_{0}^{t}\int_{\mathbb{R}_0}z_i(N_i(ds,dz_i)-\nu_i(dz_i)ds), \quad t \geq 0. 
$$
Moreover, we assume that three processes $W$, $Z$, and $\theta$ are mutually independent and adapted to the filtration $(\mf_t)_{t\geq 0}$; for each $\theta \in S$, $b(\theta,.)$, $\sigma(\theta, .)$ and $\gamma(\theta, .)$ are measurable functions. 

Stochastic differential equations (SDEs) have been used to model many random processes in biology, engineering, finance, and physics (see \cite{KP, MT}). In particular, SDEs driven by jumps with switching coefficients can be used to capture sudden changes in the dynamic of these processes (see \cite{Jobert, Smith, MaoYuanfirst, Cont, Oksendal, NOP}).
In these applications, one often wants to understand the long-time behavior of the systems by computing their invariant distribution. The partial differential equations approach may help to find such distributions for low dimensional SDEs with regular coefficients. One should use some simulation method for high dimensional SDEs with less regular coefficients (see \cite{LMS, MSH, Talay, TT}). More precisely, let $\pi$ be the invariant measure, then for some  function $\varphi$, 
$$\pi(\varphi):= \int \varphi(x) \pi (dx) \approx \mathbb{E}[\varphi(X_T)] \text{ for } T \text{ large enough}.$$
Since the value of $\mathbb{E}[\varphi(X_T)] $ is again not analytically tractable, one needs to approximate $X_T$ by $\hat{X}_T$ which can be simulated from equation \eqref{eqn1} by a time discretization scheme. It is well-known that for classical approximation schemes, such as the Euler-Maruyama (EM) scheme, the error of the estimate depends on $T$ and may go to $\infty$ as $T \to \infty$.
Therefore, it would be useful to construct a numerical approximation scheme for $X_T$ such that the error of the estimate does not depend on $T$. 

An extensive body of literature exists on numerical SDE methods for SDEs (see \cite{KP, MT, KB}). Additionally, significant attention has been given recently to numerical methods for SDEs with non-Lipschitz coefficients. In particular, it has been demonstrated in \cite{HJKa} that for SDEs with super-linear growth coefficients, the classical Euler-Maruyama scheme may fail to converge in the $L^p$-norm. To address this issue, various modified Euler-Maruyama schemes have been developed, including the tamed Euler-Maruyama scheme (\cite{HJKb, Sabanis1, Sabanis2, DKS}) and the truncated Euler-Maruyama scheme (\cite{Mao}). These schemes have been further developed for SDE with jumps and  Markovian switching in \cite{Deng, KS2, HLM, Nguyen}
In \cite{FG}, Fang and Giles introduced an adaptive EM method and showed its convergence in $L^2$-norm on the whole time interval $[0,\infty)$ when applying for a class of SDEs whose drift coefficient is polynomial growth Lipschitz continuous and diffusion coefficient is bounded and globally Lipschitz continuous. 
In  \cite{KLN, KLNT, LLNT, MKV}, the tamed-adaptive EM scheme was proposed by combining the ideas of the adaptive scheme in \cite{FG} with the tamed scheme introduced in \cite{HJKb}. It has been proven in \cite{KLN, KLNT, LLNT, MKV} that the tamed-adaptive EM scheme also converges in $L^2$-norm on the whole time interval $[0,\infty)$ when applying for a class of SDEs whose drift and diffusion coefficients are locally Lipschitz and superlinearly growth.  
Furthermore, the tamed-adaptive Euler-Maruyama scheme was introduced in \cite{KLN, KLNT, LLNT, MKV} by combining the adaptive scheme from \cite{FG} with the tamed scheme from \cite{HJKb, Sabanis1, Sabanis2}. It has been shown in \cite{KLN, KLNT, LLNT, MKV} that the tamed-adaptive Euler-Maruyama scheme also converges in the 
$L^2$-norm over the entire interval 
$[0,\infty)$ when applying for a class of SDEs whose drift and diffusion coefficients are locally Lipschitz and superlinearly growth. 

The purpose of this paper is to employ the tamed-adaptive Euler-Maruyama (TAEM) approximation scheme to construct an approximation for the invariant distribution of equation \eqref{eqn1}. Our approach is inspired by the works \cite{FG, KLN, KLNT, LLNT, MKV}. By adapting the methodology in \cite{Giles, FG}, we consider a Multi-level Monte Carlo scheme to estimate $\pi(\varphi)$, where $\pi$ represents the invariant distribution of the process $X$. This scheme applies to any time-discretization approximation of $X$ that satisfies a uniform-in-time estimate for the $L^2$-norm of the error (see condition \eqref{MLMC1}).
Subsequently, we present the TAEM scheme and demonstrate its strong convergence over both finite and infinite time intervals. Notably, this is the first paper to study long-time approximation for multidimensional Markovian switching SDEs with superlinear growth diffusion coefficients, even in the case of SDEs without jump components.

The structure of this paper is as follows: Section \ref{Sec:2} introduces a set of conditions on the coefficients of equation \eqref{eqn1} along with moment estimates for the exact solution. 
Section \ref{InvaMeasAppro} establishes the existence and uniqueness of the invariant measure and then proposes a Multi-level Monte Carlo approximation scheme for it. In Section \ref{sec:main}, we present the TAEM scheme, analyzing its convergence over both finite and infinite time intervals. 
Finally, Section \ref{sec:nume} provides a numerical analysis of the TAEM scheme.

\section{Model assumptions and moment estimates} \label{Sec:2}
For $x \in \mathbb{R}$, $\lceil x \rceil$ is the smallest integer greater than or equal to $x$. 
For vectors $u, v \in \mathbb{R}^d$, we denote by $|v|$ the Euclidean norm of $v$, and by $\langle v,w\rangle$ the inner product of $u$ and $v$. 
For a square matrix $A$, we denote by $A^\mathsf{T}$ its transpose and by $|A|$ its  Frobenius  norm. 
For each $L>0$ and $m\geq 0$, let $\textbf{Lip}(L,m)$ denote the set of all functions $f$ defined on $\mathbb{R}^d$, such that 
$|f(x) - f(y)| \leq L(1+ |x|^m + |y|^m)|x-y|$.

Throughout this paper, we always  assume that the following conditions on  the coefficients $b, \sigma, \gamma$ and the L\'evy measure $\nu$ hold. 
	\begin{enumerate}[\indent \bf C1.] 
		\item  
  There exist positive constants $L_1$ and  $l\geq 1$ such that for any $i \in S$, $b(i,.) \in \textbf{Lip}(L_1,l)$.
		\item  There exist positive constants $L_2$  and  $m\geq 1$  such that  for any $i \in S$, $\sigma(i,.) \in \textbf{Lip}(L_2,m)$.
		\item There exists a positive constant $L_3$ such that for any $i \in S$, $\gamma(i,.) \in \textbf{Lip}(L_3,0)$. 
		\item For all $1\leq p\leq q$, it holds that $\int_{\vert z\vert >1} \vert z\vert^p\nu(dz)<\infty$,  and $\int_{0< \vert z\vert \le 1} \vert z\vert \nu(dz)<\infty$.
		\item There exist $p_0 \geq 2$, $\zeta_0 \in \mathbb R$, and $\zeta_1 \in [0, +\infty)$ such that 
		$$ \langle x,b(i,x) \rangle  + \dfrac{p_0-1}{2}  |\sm(i,x)|^2 +\textcolor{black}{\dfrac{1}{2L_0}|\gamma(i,x)|^2 \int_{\bR_{0}^d}|z|\left(\left(1+L_{0}|z|\right)^{p_0-1}-1\right) \nu(d z)} \leq {\zeta_0}  |x|^2+{\zeta_1},$$
  for any $i \in S$,  $x \in \br^d$ and $L_0 = \max\{3L_3, \sum_{i = 1}^N|\gamma(i,0)|\}$. 
		\item  There exist constants $\epsilon>0$ and  $\alpha \in \bR$ such that for any $i \in S; x,y \in \bR^d$, it holds that 
		\begin{align*} 
		2\langle x-y, b(i,x)-b(i,y)\rangle+ &\left(1+\epsilon\right) |\sm(i,x)-\sm(i,y)|^{2}+ |\gamma(i,x)-\gamma(i,y)|^2 \int_{\bR_{0}^d}  |z|^{2} \nu(d z) \leq \alpha |x-y|^{2}.
		\end{align*}
	\end{enumerate}

\begin{Rem}\label{rem2.1}
	\begin{enumerate}[\indent (i)]
\item Conditions \textbf{C1}--\textbf{C4} are sufficient to guarantee that equation \eqref{eqn1} has a unique strong solution 	(see Theorem 2 in \cite{GK80} or Theorem III.2.32 in \cite{JS03}).
  \item When the coefficient $\sigma$ is  Lipschitz, that is, $m=0$ in \textbf{C2},   we can choose $\epsilon=0$ in \textbf{C6}, which will be seen in the proof of Theorem \ref{dinh ly 6}.
		\item When the coefficient  $\gamma$ satisfies Condition \textbf{C3}, then for any $i \in S ;x \in \br^d$, it holds that 
		\begin{equation}
			|\gamma(i,x)| \leq L_0 (1+|x|), \quad \text{ with } L_0 = \max\left\{3L_3, \sum_{i \in S}|\gamma(i,0)|\right\}.  \notag
		\end{equation}
		\item It follows from  Condition \textbf{C5} that,  for \textcolor{black}{any $i \in S; x \in \mathbb{R}^{d}$ and} $p\in[2,p_0]$,
		$$ \langle x,b(i,x) \rangle  + \dfrac{p-1}{2}  |\sm(i,x)|^2 +\textcolor{black}{\dfrac{1}{2L_0}|\gamma(i,x)|^2 \int_{\bR_{0}^d}|z|\left(\left(1+L_{0}|z|\right)^{p-1}-1\right) \nu(d z)} \leq {\zeta_0}  |x|^2+{\zeta_1}.$$
	\end{enumerate}
\end{Rem}

We provide the following results on the moment of the exact solution, which might be of independent interest.   
\begin{Prop}\label{nghiem dung 1}
	Assume that $X=(X_t)_{t \geq 0}$ is a solution to equation \eqref{eqn1}, 
 $\sigma$ is bounded on every compact subset of \ $S \times \mathbb{R}^d$, and \textbf{C4} holds for $q=2p_0$.
	Then, for any $p \in (0,p_0]$ and $t \ge 0$, there exists a positive constant $C_p$ which does not depend on $t$ such that
	\begin{align*} 
	\bE \left[\vert X_t\vert^p\right] \le \begin{cases}  C_p (1+e^{{\zeta_0} pt}) & \text{ if \;\;} {\zeta_0} \not = 0,\\
	C_p (1+t)^{p/2} & \text{ if \;\;} {\zeta_0}  = 0. \end{cases} 
	\end{align*}
\end{Prop}
\noindent Note that when ${\zeta_0}<0$, we get that $\sup_{t \geq 0} \bE \left[\vert X_t\vert^p\right] \le 2C_p$.
The proof of Proposition \ref{nghiem dung 1} will be skipped since it is similar to the one of Proposition 2.5 in \cite{MKV}.

\section{Adaptive Multi-level Monte Carlo for invariant distribution}\label{InvaMeasAppro}

\subsection{Existence and uniqueness of the invariant measure}
In this section, we shall prove  the existence and uniqueness of the invariant measure of the class of Markovian switching L\'evy-driven SDEs  (\ref{eqn1}) with super-linearly growth coefficients.

The following lemma shows how the solution of $\eqref{eqn1}$ depend on its initial value. 
\begin{Lem}\label{SDEcontract} 
    Assume that  $\sigma$ is bounded on every compact subset of $\mathbb{R}^d \times S$, and \textbf{C4} holds for $q = 2p_0$. Let $X_t$ and $Y_t$ be two solutions of \eqref{eqn1}. Then,  for any $ t > 0,$ it holds that 
    $$
    \mathbb{E}[|X_t - Y_t|^2] \leq e^{\alpha t}  \mathbb{E}[|X_0 - Y_0|^2].$$
\end{Lem}
\begin{proof}
    Let $e_t = X_t - Y_t$. It follows from It\^o's formula that,  for any $t \geq 0$, 
    \begin{align*}
        &e^{-\alpha t} |e_t|^2
         = |e_0|^2 - \int_0^t \alpha e^{-\alpha s} |e_s|^2 ds + \int_0^t 2 e^{-\alpha s} \left \langle e_s, (\sigma(\theta_s, X_s) -  \sigma(\theta_s, Y_s)) dW_s \right \rangle \notag \\
        & \quad + \int_0^t e^{-\alpha s} \Big( 2 \left \langle e_s, b(\theta_s, X_s) -  b(\theta_s, Y_s)\right \rangle + |\sigma(\theta_s, X_s) -  \sigma(\theta_s, Y_s)|^2  \\
        & \qquad  + |\gamma(\theta_{s}, X_{s}) - \gamma(\theta_{s}, Y_{s})|^2 \int_{\mathbb{R}_0^d} |z|^2 \nu (dz)\Big) ds \notag \\
        & \quad + \int_0^t \int_{\mathbb{R}_0^d} e^{-\alpha s} (\left| e_{s-} + \left( \gamma(\theta_{s-}, X_{s-}) - \gamma(\theta_{s-}, Y_{s-})  \right) z \right|^2 - |e_{s-}|^2) \widetilde{N}(ds, dz). 
    \end{align*}
    It follows from  \textbf{C6} that $ \mathbb{E}[e^{-\alpha t} |e_t|^2] \leq \e[|e_0|^2]$, which implies the desired result. 
\end{proof}
Before the statement and proof of the next theorem, we shall introduce some notation and recall some essential properties.\\
Let $X_t(x)$ denote process $X$ with $X_0 = x$. Also, for some $B \in \mathcal{B}(\mathbb{R}^d)$, we define $p_x(t,B) := \mathbb{P}(X_t(x) \in B)$.
By the Markov property, we have the Chapman–Kolmogorov equation
\begin{equation}\label{Chap-Kol}
    p_x(t + s, B) = \int_{\mathbb{R}^d} p_x(t,dy)p_y(s,B).
\end{equation}
    
The next theorem will state the existence and uniqueness of the invariant measure of SDE (\ref{eqn1}).
\begin{Thm} \label{exi-uni inva}
    Assume that Conditions \textbf{C1 - C6} hold with $\zeta_0 < 0$,  $\alpha<0$, $p_0 \ge 2$, $\sigma$ is bounded on every compact subset of $\mathbb{R}^d \times S$, and \textbf{C4} holds for $q = 2p_0$. Then there exists a unique invariant measure of the SDE (\ref{eqn1}).
\end{Thm}
\begin{proof}
Firstly, we shall prove the uniqueness of the invariant measure. Assume that there exist two invariant measures $\nu$ and $\mu$. Let $P_t\varphi (x) = \bE[\varphi(X_t)|X_0 = x] = \int_{\mathbb{R}^d} \varphi(y) p_x(t, dy)$, where $\varphi \in C^2_b$. It follows from Lemma \ref{SDEcontract} that 
\begin{align*}
    \left|\int \varphi(x)(\nu - \mu)(dx)\right| = \left|\int P_t\varphi(x)(\nu - \mu)(dx)\right| \leq ||\triangledown \varphi||_{\infty}  e^{\alpha t/2}  \mathcal{W}_1(\nu, \mu),
\end{align*}
where
$
||f||_{\infty} = \sup_x |f(x)|, \ \mathcal{W}_1(\nu, \mu) = \sup_{||\triangledown f|| \leq 1} \left| \int f(x)(\nu - \mu)(dx) \right|.
$
Let $t \rightarrow
+ \infty$, we have that $\left|\int \varphi(x)(\nu - \mu)(dx)\right| = 0$ for any $\varphi \in C^2_b$, which implies $\nu = \mu$.

Secondly, we prove that there exists such an invariant measure. Let $\mathcal{L}_i$ be the infinitesimal generator of $X$ where $\theta = i$. 
Consider the function $V(x) = |x|^2 \geq 0$ and
\begin{align*}
    \mathcal{L}_i V(x) 
    &= \langle V_x(x), b(i,x) \rangle + \dfrac{1}{2} V_{xx}(x) \times \text{trace}[\sigma^\top\sigma(i,x)] \\
    & \quad + \int_{\mathbb{R}_0^d} (V(x + \gamma(i,x)z) - V(x) - \langle V_x(x), \gamma (i,x)z \rangle ) \nu(dz),
\end{align*}
for any $x \in \mathbb{R}^d$ and $i \in S$.
It follows from Condition \textbf{C5} that 
\begin{align*}
    \mathcal{L}_i V(x) 
    & \leq 2 \left \langle x, b(i,x) \right \rangle + \left|\sigma(i,x)\right|^2 + \int_{\mathbb{R}_0^d} |\gamma (i,x)z|^2 \nu(dz)
     \leq 2\zeta_0 |x|^2 + 2\zeta_1 \leq 2\zeta_1.
\end{align*}
This implies that there exists a constant $C$  such that $\sup_{i \in S}\sup_{x \in \mathbb{R}^d} \mathcal{L}_i V(x) \leq C$. Moreover, for any $M > 0$, there exists a compact set $K_M$ so that $\sup_{i \in S}\sup_{x \in K_M^C}  \mathcal{L}_i V(x) \leq -M$  where $K_M^C:= \mathbb{R}^d \backslash K_M$.

Applying the It\^o's formula for $V(X_t(x))$, then taking expectation both sides, we obtain
\begin{align*}
    0 
    &\leq \mathbb{E}[V(X_t(x))] = V(x) + \mathbb{E}\left[\int_0^t \sum_{i \in S} \mathcal{L}_i V(X_s(x)) \mathbb{I}_{\{\theta_s = i\}} ds\right]\\
    & = V(x) + \mathbb{E}\left[\int_0^t \mathbb{I}_{K_M^C}(X_s(x))  \sum_{i \in S} \mathcal{L}_i V(X_s(x)) \mathbb{I}_{\{\theta_s = i\}}ds + \int_0^t \mathbb{I}_{K_M}(X_s(x)) \sum_{i \in S} \mathcal{L}_i V(X_s(x))ds \right]\\
    & \leq V(x) + \mathbb{E}\left[\int_0^t \sum_{i \in S} (-M)\;\mathbb{I}_{K_M^C}(X_s(x)) \mathbb{I}_{\{\theta_s = i\}}ds + \int_0^t \sum_{i \in S} C \; \mathbb{I}_{K_M}(X_s(x)) \mathbb{I}_{\{\theta_s = i\}}ds \right]\\
    & = V(x) + \mathbb{E}\left[\int_0^t \sum_{i \in S} (C + M)\;\mathbb{I}_{K_M}(X_s(x)) \mathbb{I}_{\{\theta_s = i\}}ds - \int_0^t \sum_{i \in S} M \mathbb{I}_{\{\theta_s = i\}}ds \right]\\
    & \leq V(x) + (C+M) \int_0^t \mathbb{P}(X_s(x) \in K_M)ds - Mt.
\end{align*}
Thus, for some fixed initial value $x \in \mathbb{R}^d$, for all $M>0$, there exists a compact set $K_M$ satisfying
\begin{align}\label{CondTightness}
\dfrac{-V(x)}{t  (C + M)} + \dfrac{M}{C + M} \leq \dfrac{1}{t} \int_0^t \mathbb{P}(X_s(x) \in K_M)ds.
\end{align}
Let $t_n \uparrow \infty$ be a sequence of time points. For each $n$, we define a Borel measure $\mu_n$ on $\mathbb{R}^d$ as
$$
\mu_n(K) := \frac{1}{t_n} \int_0^{t_n} \mathbb{P}(X_s(x) \in K)ds.
$$
From (\ref{CondTightness}), we have that
$$
\lim_{M \rightarrow \infty} \sup_{n} \dfrac{1}{t_n} \int_0^{t_n} \mathbb{P}(X_s(x) \in K_M)ds = 1.
$$
Thus, for any $\varepsilon > 0$, there exist $n_0$ and $M_0$ such that for any $n > n_0$,
$$
\dfrac{1}{t_n} \int_0^{t_n} \mathbb{P}(X_s(x) \in K_{M_0})ds > 1 - \varepsilon.
$$
It follows from Prokhorov's theorem that there exists a subsequence $\{\mu_{n_j}\}$ that converges weakly to a probability measure $\mu_0$. In what follows, we shall prove that $\mu_0$ is an invariant measure for $X$.

For any $f \in C^1_b$, we have
\begin{equation}\label{Converge}
    \int (P_t f)(y) \mu_0(dy) = \lim_{j \rightarrow \infty} \int (P_t f)(y) \mu_{n_j}(dy).
\end{equation}
For some fixed $n_j$, by Fubini's theorem and from (\ref{Chap-Kol}), we have
\begin{align*}
    \int (P_t f)(y) \mu_{n_j}(dy) 
    &= \int \dfrac{1}{t_{n_j}} \int_0^{t_{n_j}} (P_t f)(y) p_x(s, dy) ds = \dfrac{1}{t_{n_j}} \int_0^{t_{n_j}} \int \int f(z) p_y(t,dz) p_x(s, dy) ds\\
    &= \dfrac{1}{t_{n_j}} \int_0^{t_{n_j}} \int f(z) p_x(t + s,dz) ds = \dfrac{1}{t_{n_j}} \int_0^{t_{n_j}} (P_{t+s}f)(x) ds.
\end{align*}
By change of variables, the last expression above can be written as
\begin{align}\label{Change-Vari}
    \dfrac{1}{t_{n_j}} \int_0^{t_{n_j}} (P_{t+s}f)(x) ds 
    &= \dfrac{1}{t_{n_j}} \int_t^{t_{n_j} + t} (P_{u}f)(x) du\notag \\ 
    &=\dfrac{1}{t_{n_j}} \left[\int_0^{t_{n_j}} (P_{u}f)(x) du + \int_{t_{n_j}}^{t_{n_j} + t} (P_{u}f)(x) du - \int_0^{t} (P_{u}f)(x) du \right].
\end{align}
Thus, from (\ref{Converge}) and (\ref{Change-Vari}), we have that
\begin{align*}
    \int (P_t f)(y) \mu_0(dy) 
    &= \lim_{j \rightarrow \infty} \dfrac{1}{t_{n_j}} \left[\int_0^{t_{n_j}} (P_{u}f)(x) du + \int_{t_{n_j}}^{t_{n_j} + t} (P_{u}f)(x) du - \int_0^{t} (P_{u}f)(x) du \right]\\
    &= \lim_{j \rightarrow \infty} \dfrac{1}{t_{n_j}} \int_0^{t_{n_j}} (P_{u}f)(x) du = \lim_{j \rightarrow \infty} \int f(y) \mu_{n_j}(dy) = \int f(y) \mu_0(dy),
\end{align*}
where the last equality is due to the weak convergence of $\{\mu_{n_j}\}$. Therefore, $\mu_0$ is an invariant measure. This concludes the proof. 
\end{proof}

\subsection{Approximation algorithm for the invariant measure}
We consider the approximation for 
$   \pi(\varphi) := \mathbb{E}_{\pi}(\varphi) = \int_{\mathbb{R}^d} \varphi(x) \pi(dx),$
where $\pi$ is the invariant measure of the SDE (\ref{eqn1}) and $\varphi$ is some measurable function from $\mathbb R^d$ to $\mathbb R$. This quantity will be approximated by the Multi-level Monte Carlo method as follows. First, we suppose that for each $l=0,1,\ldots$, $X_T$ can be numerically approximated by $\widehat{X}^{l}_T$ such that 
\begin{equation} \label{MLMC1}
  \mathbb{E}[|\widehat{X}^l_T - X_T|^2] \leq \mathbf{K}_0 M^{-l},\quad l = 0,1,\ldots 
\end{equation} 
for some fixed positive  constants $\mathbf{K}_0$ and $M$, which depend on neither  $l$ nor  $T$.
We can write 
$$
\pi(\varphi) \approx \mathbb{E}[\varphi(X_T)] \approx  \mathbb{E}[\varphi(\widehat{X}^L_T)] = \mathbb{E}[\varphi(\widehat{X}^0_T)] + \sum_{l = 1}^L \mathbb{E}[\varphi(\widehat{X}^l_T) - \varphi(\widehat{X}^{l-1}_T)].
$$
Then, the Multi-level Monte Carlo (MLMC) estimator is the following telescoping sum:
\begin{equation} \label{MLMC estimator}
    \mathbb{E}[\varphi(\widehat{X}^L_T)] \approx \widehat{Y} := \dfrac{1}{N_0} \sum_{n=1}^{N_0} \varphi(\widehat{X}^{(n,0)}_T) + \sum_{l=1}^L\left\{\dfrac{1}{N_l} \sum_{n=1}^{N_l} (\varphi(\widehat{X}^{(n,l)}_T) - \varphi(\widehat{X}^{(n,l-1)}_T))\right\},
\end{equation}
where 
$(\widehat{X}^{(n,l)}_T)_{1 \leq n \leq N_l}$  are $N_l$ independent copies of $\widehat{X}^l_T$, and for each $n$, $\widehat{X}^{(n,l)}_T$ and $\widehat{X}^{(n,l-1)}_T$ must be generated by the same Brownian and jumps path. Moreover, if we denote by $\overline{C}_l$ the expected computational cost of a sample on level $l$, then we may deduce that the expected overall computational cost of the estimation will be
$$
    \mathbf{C}_{MLMC} := \sum_{l=0}^L N_l \overline{C}_l.
$$

\begin{Lem}\label{expconv}
Let  $\varphi \in \mathbf{Lip}(L_\varphi, 0)$. 
Assume that  $\widehat{X}^{l}_T$ satisfies \eqref{MLMC1}.   Then there exists a constant $\mathbf{K}_1 > 0$ depending on $x_0$ and $L_{\varphi}$ such that
    $$
    |\mathbb{E}\left[\varphi(X_t) - \pi(\varphi)\right]| \leq \mathbf{K}_1 e^{\alpha t/2}.
    $$
\end{Lem}
\begin{proof}
Let $Y_0$ be a random variable with density $\pi$ and $Y_t$ be the solution of \eqref{eqn1} with initial value $Y_0$. Since $Y_t$ also has density $\pi$,  thanks to  Lipschitz property of $\varphi$, Hölder's inequality, Lemma \ref{SDEcontract} and Proposition \ref{nghiem dung 1}, we have 
    \begin{align*}
        \left|\mathbb{E}\left[\varphi(X_t) - \pi(\varphi)\right]\right| 
        &= \left|\mathbb{E}\left[\varphi(X_t) - \varphi(Y_t)\right]\right|  \leq L_{\varphi} e^{\alpha t/2} \mathbb{E}[|X_0 - Y_0|^2]^{1/2}\leq  \mathbf{K}_1 e^{\alpha t/2},
    \end{align*}
    where $\mathbf{K}_1 = 2L_{\varphi}  [|x_0| + \sqrt{\mathbb{E}[|Y_0|^2]}]$.
\end{proof}

\begin{Lem}\label{variance bound}
    Let  $\varphi \in \mathbf{Lip}(L_\varphi, 0)$.  Assume that  $\widehat{X}^{l}_T$ satisfies \eqref{MLMC1}.  Then, there exists a positive constant $\mathbf{K}_2$ such that, for all level $l$, 
    $$
   \mathrm{Var}[\varphi(\widehat{X}^l_T) - \varphi(\widehat{X}^{l-1}_T)] \leq \mathbf{K}_2 M^{-l}.
    $$
\end{Lem}
\begin{proof}
    By the Lipschitz property of  $\varphi$, we have
    $$
    \text{Var}[\varphi(\widehat{X}^l_T) - \varphi(\widehat{X}^{l-1}_T)] \leq \mathbb{E}[|\varphi(\widehat{X}^l_T) - \varphi(\widehat{X}^{l-1}_T)|^2] \leq L_{\varphi}^2 \; \mathbb{E}[|\widehat{X}^l_T - \widehat{X}^{l-1}_T|^2],
    $$
    where $\widehat{X}^l_t$ and $\widehat{X}^{l-1}_t$ share the same Brownian motion from $0$ to $T$. Applying the simple inequality $(a + b)^2 \leq 2(a^2 + b^2)$, which holds for all real numbers $a$ and $b$, we obtain that
    $$
        \mathbb{E}[|\widehat{X}^l_T - \widehat{X}^{l-1}_T|^2] \leq 2(\mathbb{E}[|\widehat{X}^l_T - X_T|^2] + \mathbb{E}[|X_T - \widehat{X}^{l-1}_T|^2]).
    $$
It follows from \eqref{MLMC1} that 
    $        \mathbb{E}[|\widehat{X}^l_T - X_T|^2] \leq \mathbf{K}_0 M^{-l}$, and $\mathbb{E}[|X_T - \widehat{X}^{l-1}_T|^2] \leq \mathbf{K}_0 M^{-l+1}.$
Let  $\mathbf{K}_2 = 2L_\varphi^2 \mathbf{K}_0 (1+M)$, we have 
    $$
       \text{Var}[\varphi(\widehat{X}^l_T) - \varphi(\widehat{X}^{l-1}_T)] \leq 2L_{\varphi}^2 (\mathbb{E}[|\widehat{X}^l_T - X_T|^2] + \mathbb{E}[|X_T - \widehat{X}^{l-1}_T|^2]) \leq \mathbf{K}_2 M^{-l},
    $$
which is the desired result.     
    \end{proof}

\begin{Thm}
    Let $\varepsilon$ be some positive constant, and (\ref{MLMC1}) holds. Let 
    $$T = \left \lceil \dfrac{1}{\alpha} \log \left(\dfrac{\varepsilon^2}{6 \mathbf{K}_1^2}\right) \right \rceil, \quad  L = \left \lceil \dfrac{2|\log \varepsilon| + \log(6\mathbf{K}_0 L_{\varphi}^2)}{\log M} \right \rceil, \quad \text{and } N_l = \left\lceil \frac{3 \mathbf{K}_2 (L+1)}{\varepsilon^2 M^{l}} \right\rceil, $$
    where $\mathbf{K}_0$, $\mathbf{K}_1$, and $\mathbf{K}_2$ are given in \eqref{MLMC1}, Lemma \ref{expconv} and Lemma \ref{variance bound}, respectively.  Then  the MLMC estimator (\ref{MLMC estimator}) has a mean square error  satisfying 
    $$
        \mathbb{E}[(\widehat{Y} - \pi(\varphi))^2] \leq  \varepsilon^2.
    $$
\end{Thm} \label{MSEbound}
\begin{proof}
    We can decompose the mean square error as
    \begin{align}
        \mathbb{E}[(\widehat{Y} - \pi(\varphi))^2]
        &= \mathbb{E}[((\widehat{Y} - \mathbb{E}[\widehat{Y}]) + (\mathbb{E}[\widehat{Y}] - \pi(\varphi)))^2]  \notag \\
        &\leq \text{Var}[\widehat{Y}] + 2|\mathbb{E}[\widehat{Y}] - \mathbb{E}[\varphi(X_{T})]|^2 + 2|\mathbb{E}[\varphi(X_{T})] - \pi(\varphi)|^2.\label{require3}
    \end{align}
    We shall bound each term in the last expression  by $\varepsilon^2 / 3.$
    Indeed, from Lemma \ref{variance bound}, we have that
    \begin{equation}\label{varymu}
        \text{Var}[\widehat{Y}] = \sum_{l=0}^L N_l^{-1}  \mathrm{Var}[\varphi(\widehat{X}^l_T) - \varphi(\widehat{X}^{l-1}_T)] \leq \frac 13 \varepsilon^2.
    \end{equation}
    Next, by the Lipschitz property of $\varphi$, \eqref{MLMC1}, and  Jensen's inequality, we have
    \begin{align}
        2|\mathbb{E}[\widehat{Y}] - \mathbb{E}[\varphi(X_{T})]|^2 
        &= 2|\mathbb{E}[\varphi(\widehat{X}^L_{T}) - \varphi(X_{T})]|^2 \notag \\
        &\leq 2 L_{\varphi}^2 \mathbb{E}[|\widehat{X}^L_{T} - X_{T}|^2] \leq 2 L_{\varphi}^2 \mathbf{K}_0 M^{-L}  \leq \frac 13 \varepsilon^2. \label{require2}
    \end{align}
    Also,  it follows from Lemma \ref{expconv}  that
    \begin{equation} \label{require1}
        2|\mathbb{E}[\varphi(X_{T})] - \pi(\varphi)|^2 \leq 2 \mathbf{K}_1^2 e^{\alpha T} \leq \frac 13 \varepsilon^2.
    \end{equation}
    Therefore, from (\ref{require3}), (\ref{varymu}), (\ref{require2}), and (\ref{require1}), we conclude our proof.
\end{proof}

In the following section, we propose an approximation scheme for the solution of the SDE (\ref{eqn1}), which satisfies the Condition \eqref{MLMC1}. We will show that for a fixed  precision $\varepsilon$, the expected cost $\mathbf{C}_{MLMC}$ is proportional to $\varepsilon^2 |\log \varepsilon|^3$.


\section{TAEM scheme} \label{sec:main} 

\subsection{Construction  of  the TAEM scheme}
For each $\Delta \in (0,1)$,  the TAEM scheme of equation \eqref{eqn1} is defined by
\begin{equation}\label{EM1}
\begin{cases} 
&t_0=0, \quad \widehat{X}_0=x_0, \quad t_{k+1}=t_k+h(\widehat{X}_{t_k})\Delta,\\ 
&  \widehat{X}_{t_{k+1}}= \widehat{X}_{t_k}+b(\theta_{t_k},\widehat{X}_{t_k})\left(t_{k+1}-t_k\right)+\sm_{\Delta}(\theta_{t_k},\widehat{X}_{t_k})(W_{t_{k+1}}-W_{t_{k}}) + \gamma_{\Delta}\left(\theta_{t_k},\widehat{X}_{t_k}\right)\left(Z_{t_{k+1}}-Z_{t_k}\right),
\end{cases} 
\end{equation}
where 
\begin{equation} \label{chooseh} 
h(x)=\dfrac{h_0}{(1+\sum_{i\in S}|b(i,x)|^2+\sum_{i\in S}|\sm(i,x)|+|x|^l)^2+(\sum_{i\in S}|\gamma(i,x)|)^{p_0}}. 
\end{equation}
Note that $h_0$ is some positive constant;  $l$ and $p_0$ are respectively defined in Conditions  \textbf{C1} and  \textbf{C5}.   Moreover, the coefficients $\sigma_\Delta=(\sigma_{\Delta,ij})_{1\leq i,j\leq d}: S \times \mathbb{R}^d \to \mathbb{R}^d\otimes \mathbb{R}^d$ and $\gamma_{\Delta}=(\gamma_{\Delta,ij})_{1\leq i,j\leq d}: S \times \mathbb{R}^d \to \mathbb{R}^d\otimes \mathbb{R}^d$ are certain approximations of $\sigma$ and $\gamma$.
Moreover, we suppose that there exists  positive constant $L$  such that, for any $i \in S$, and $x \in \mathbb{R}^d$, it holds that 
\begin{equation} \label{T3T4}
|\sm_{\Delta}(i,x)| \le L (|\sigma(i,x)| \wedge \Delta^{-1/2}), \quad |\gamma_\Delta (i,x)| \le L (|\gamma(i,x)| \wedge \Delta^{-1/2}),  \quad |b(i,x)| |\gamma_{\Delta}(i,x) |  \le L/\sqrt{\Delta}.
\end{equation} 
In the following, we choose $h_0 = 1$ to simplify our presentation of the expressions.  All the following results hold for any $h_0 >0$. In practice, one can turn $h_0$ to control the number of time discretization points of the TAEM scheme.

We now provide a sufficient condition to guarantee that $t_k \uparrow \infty$ as $k \uparrow \infty$, which deduces that the TAEM approximation scheme \eqref{EM1} is well-defined.
\begin{Prop}\label{dinh ly 4}
Assume that Condition \textbf{C4} holds for $p=2$.
Then, we have
	$	\lim\limits_{k \to +\infty} t_k =+\infty \quad \text{a.s.} $ 
 \end{Prop}

\begin{proof}
    Let $\beta = \max\{m, l\} + 1. $
    For each $H > 0$ and $k \in \mathbb{N}$, we define an approximation scheme as follows
    \begin{equation*}
	\begin{cases} 
    &t^H_0 = 0, \quad t^H_{k+1}=t^H_k+h_{\Delta}^H\left(\widehat{X}^H_{t^H_k}\right), \\
	&\widehat{X}_{t^H_{k+1}}^{H} = \widehat{X}_{t^H_{k}}^{H} + b_H \left(\theta_{t^H_k},\widehat{X}_{t^H_{k}}^{H}\right) h_{\Delta}^H\left(\widehat{X}^H_{t^H_k}\right)+\sm_{\Delta} \left(\theta_{t^H_k}, \widehat{X}_{t^H_{k}}^{H}\right) (W_{t^H_{k+1}}-W_{t^H_k})+\gamma_{\Delta} \left(\theta_{t^H_k}, \widehat{X}_{t^H_{k}}^{H}\right) (Z_{t^H_{k+1}}-Z_{t^H_k}),
	\end{cases}
	\end{equation*}
    where
    \begin{equation*}
        b_H(i,x) = 
        \begin{cases}
            b(i,x) \quad &\text{if } |x|^{\beta} \leq H\\
            \dfrac{x}{1 + |x|^2} \quad &\text{if } |x|^{\beta} > H
        \end{cases} \text{ and }
        h_{\Delta}^H(x) =
        \begin{cases}
        \Delta h\left(x\right) & \text { if } |x|^{\beta} \leq H \\ \frac{\Delta h_0}{L(1+H)} & \text { if } |x|^{\beta} > H\end{cases}.
    \end{equation*}
    
    For each $k\in\mathbb{N}$, $t^H_k$ is a stopping time and  $t^H_{k+1}$ is $\mathcal{F}_{t^H_k}$-measurable. We define $\underline{t}^H: = \max\{t^H_k: t^H_k \leq t\}.$ Thus, $\underline{t}^H $ is also a stopping time. Then, $\widehat{X}^H=(\widehat{X}^H_t)_{t \geq 0}=(\widehat{X}^H_{1,t},\widehat{X}^H_{2,t},\ldots,\widehat{X}^H_{d,t})_{t \geq 0}$ is the solution to the following SDE
    \begin{equation*}
    d \widehat{X}^H_{t}= b_H (\theta^H_{\underline{t}^H},\widehat{X}^H_{\underline{t}^H}) dt+\sm_{\Delta} \left(\theta^H_{\underline{t}^H},\widehat{X}^H_{\underline{t}^H}\right)  dW_{t} +\gamma_{\Delta} \left(\theta^H_{\underline{t}^H-}, \widehat{X}^H_{\underline{t}^H-} \right)  dZ_{t}, \qquad \widehat{X}_{0} = x_0.
    \end{equation*}
    The rest of the proof can be carried out by following the argument in the proof of Proposition 4.1 in \cite{MKV}. 
\end{proof}

For each $t>0$, we define the stopping time  $\underline{t} := \max \left\{t_{n}: t_{n} \leq t\right\}$,  and   $N_{t} := \max \left\{n: t_{n} \leq t\right\}$. Moreover, let's consider 
\begin{equation}\label{EM3}
	\widehat{X}_{t}=\widehat{X}_{\underline{t}}+b\left(\theta_{\underline{t}},\widehat{X}_{\underline{t}}\right)(t-\underline{t})+\sigma_{\Delta}\left(\theta_{\underline{t}},\widehat{X}_{\underline{t}}\right)\left(W_{t}-W_{\underline{t}}\right)+\gamma_{\Delta}\left(\theta_{\underline{t}-},\widehat{X}_{\underline{t}-}\right)\left(Z_{t}-Z_{\underline{t}}\right).
\end{equation}
Thus, $\widehat{X}=(\widehat{X}_t)_{t \geq 0}$ satisfies the following equation
\begin{equation}\label{EM2}
d \widehat{X}_{t}= b (\theta_{\underline{t}},\widehat{X}_{\underline{t}}) dt+\sm_{\Delta} \left(\theta_{\underline{t}},\widehat{X}_{\underline{t}}\right)  dW_{t} +\gamma_{\Delta} \left(\theta_{\underline{t}-}, \widehat{X}_{\underline{t}-} \right)  dZ_{t}, \qquad \widehat{X}_{0} = x_0.
\end{equation}

\subsection{Moments of the TAEM scheme}

In all that follows, we choose $q=4p_0$ in Condition \textbf{C4}.
	\begin{Prop}\label{dinh ly 2}
		Assume that \textbf{C4} holds with  $q=4p_0$;    $\sigma_\Delta, \gamma_\Delta$ satisfy \eqref{T3T4}. Moreover, assume that there exist constants $\widetilde{L_0}>0$, $\widetilde{{\zeta_0}} \in \mathbb{R}$ and $\widetilde{{\zeta_1}} \in[0, +\infty)$ such that for all $i \in S, x \in \mathbb{R}^d$,
		\begin{equation}\label{cdelta}
			|\gamma_{\Delta}(i,x)|\leq \widetilde{L_0}(1+|x|),
		\end{equation}
		and
		\begin{equation} \label{T5}
			\langle x,b(i,x) \rangle  + \dfrac{p_0-1}{2}  |\sm_{\Delta}(i,x)|^2 + \dfrac{1}{2\widetilde{L_0}}|\gamma_{\Delta}(i,x)|^2 \int_{\bR_{0}^d}|z|\left(\left(1+\widetilde{L_0}|z|\right)^{p_0-1}-1\right) \nu(d z) \leq \widetilde{{\zeta_0}}  |x|^2+\widetilde{{\zeta_1}}.
		\end{equation}
Then, for any positive integer  $k \leq  p_0/2$,  there exists a positive constant $C$, which  depends neither on $t$  nor on $\Delta$, such that
		\begin{align*}
			\bE \left[|\widehat{X}_t|^{2k} \right] \vee \bE \left[|\widehat{X}_{\underline{t}}|^{2k}\right] \le \left\{ \begin{array}{l l}
				Ce^{{2k\widetilde{{\zeta_0}} t}} \quad &\text{ if \;\;} \widetilde{{\zeta_0}} >0,\\
				C(1+t)^{k} \quad &\text{ if \;\;} \widetilde{{\zeta_0}} =0,\\
				C \quad &\text{ if \;\;} \widetilde{{\zeta_0}} <0.
			\end{array} \right. 
		\end{align*}
	\end{Prop}
 The proof of Proposition \ref{dinh ly 2} follows similar arguments to those presented in the proof of Proposition 4.13 in \cite{MKV} and is therefore omitted here.

\begin{Rem}\label{nhan xet 1}
If we choose the approximated coefficients $\sm_{\Delta}, \gamma_{\Delta}$ as follows 
	\begin{align}
	&\sm_{\Delta}(i,x)= \dfrac{\sm(i,x)}{1+\Delta^{1/2}|\sm(i,x)|} \quad \text{and} \quad \gamma_{\Delta}(i,x)= \dfrac{\gamma(i,x)}{1+\Delta^{1/2}|\gamma(i,x)|(1+|b(i,x)|)}, \label{csmdelta1}
	\end{align}
	then conditions \eqref{T3T4}, \eqref{cdelta}, and \eqref{T5} hold  with $\tilde{L}_0 = L_0$, $\tilde{{\zeta_0}} = {\zeta_0}$ and $\tilde{{\zeta_1}} = {\zeta_1}$. 
\end{Rem}

\begin{Rem}\label{hq1}
	Assume that  all conditions of Proposition \ref{dinh ly 2} are valid, we have the following moment estimate of the number of timesteps $N_T$  
	\begin{equation}\label{ENT}
	\mathbb{E}\left[N_{T}-1\right] \leq C_T\Delta^{-1},
	\end{equation}
	for any $T>0$, where  $C_T>0$ is a constant and does not depend on $\Delta$.
	
	Indeed, using a similar argument as in the proof of  Lemma 2 in \cite{FG}, the estimate \eqref{ENT} can be obtained as a consequence of  Proposition \ref{dinh ly 2}. 
\end{Rem}

The difference between $\widehat{X}_t$ and $\widehat{X}_{\underline{t}}$ has the following uniform bound in time. 

\begin{Lem}\label{hq2}
	Under all assumptions of  Proposition \ref{dinh ly 2}, for any $p \in [2;p_0]$,  there exists a positive constant $C_p$, which does not  on  $\Delta$, such that
	\begin{equation*}
		\sup_{t \geq 0} \bE \left[  |\widehat{X}_t - \widehat{X}_{\underline{t}}|^p \big| \mathcal{F}_{\underline{t}}  \right] +
	\sup_{t \ge 0}\bE \left[  |\widehat{X}_t - \widehat{X}_{\underline{t}}|^p \right] \le C_p\Delta.
	\end{equation*}
\end{Lem}
\begin{proof} 
	From \eqref{EM3}, we have that for any $p \ge 2$,
	\begin{align}
	|\widehat{X}_t - \widehat{X}_{\underline{t}}|^p 
	&= \left| b(\theta_{\underline{t}},\widehat{X}_{\underline{t}})(t-\underline{t})+\sm_{\Delta}(\theta_{\underline{t}},\widehat{X}_{\underline{t}})(W_t-W_{\underline{t}}) +\gamma_{\Delta}(\theta_{\underline{t}},\widehat{X}_{\underline{t}})(Z_t-Z_{\underline{t}}) \right|^p \notag \\
	&\le 3^{p-1} \left[  \left| b(\theta_{\underline{t}},\widehat{X}_{\underline{t}})\right|^p (t-\underline{t})^p + \left|\sm_{\Delta}(\theta_{\underline{t}},\widehat{X}_{\underline{t}})\right|^p \left|W_t-W_{\underline{t}} \right|^p +\left|\gamma_{\Delta}(\theta_{\underline{t}},\widehat{X}_{\underline{t}})\right|^p \left|Z_t-Z_{\underline{t}} \right|^p \right]. \label{XtXmt-p}
	\end{align}
Note that the following moment estimates can be deduced from the strong Markov property of $W$ and $Z$: for each $r\geq 1$, there exists a positive constant $c_r$ such that, for any   
$t>0$, it holds that 
	\begin{equation*} 
	\mathbb{E}\left[\left|W_t-W_{\underline{t}^H}\right|^r\big| \mathcal{F}_{\underline{t}^H}\right] \le   c_r(t-\underline{t}^H)^{r/2};  \quad \mathbb{E}\left[\left|Z_t-Z_{\underline{t}^H}\right|^r\big| \mathcal{F}_{\underline{t}^H}\right] \le   c_r(t-\underline{t}^H).
	\end{equation*} 
These estimates together with  \textbf{C4}, \eqref{chooseh},    \eqref{T3T4}, and the Burkholder-Davis-Gundy, imply  that for any $2 \leq p \leq p_0$,
		\begin{align*} 
			\max \Bigg \{ \left|b(\theta_{\underline{s}},\widehat{X}_{\underline{s}})\right|^p (s-\underline{s})^p,  \  \mathbb{E}\left[ \left|\sm_{\Delta}(\theta_{\underline{s}},\widehat{X}_{\underline{s}})\right|^p \left|W_s-W_{\underline{s}}\right|^p\big|\mathcal{F}_{\underline{s}}\right], \mathbb{E}\left[ \left|\gamma_{\Delta} (\theta_{\underline{s}},\widehat{X}_{\underline{s}})\right|^p \left|Z_{s}-Z_{\underline{s}}\right|^p \big|\mathcal{F}_{\underline{s}}\right] \Bigg\}\leq C\Delta.
        \end{align*}
	This, together with \eqref{XtXmt-p}, concludes the desired result.
\end{proof} 

\subsection{Convergence of the TAEM scheme}

\begin{Thm}\label{dinh ly 6} 
	Assume that \textbf{C4} holds with $q = 4p_0$ and  $p_0 \geq 2\max\{ l+ 3, 2m + 2\}$. Assume that   $\sigma_\Delta$ and $\gamma_\Delta$ satisfy \eqref{T3T4}, \eqref{cdelta}, \eqref{T5}. Moreover, assume that 
 there exists a positive constant $L_4$ such that for all $i \in S; x \in \bR^d$,
	\begin{align} \label{T6}
		|\sm(i,x)-\sm_{\Delta}(i,x)| \le L_4 \Delta^{1/2}|\sm(i,x)|^2,\qquad |\gamma(i,x)-\gamma_{\Delta}(i,x)| \le L_4 \Delta^{1/2} |\gamma(i,x)|^2(1+|b(i,x)|). 
	\end{align}
Then for any $T>0$, there exist  positive constants  $K_T$ and $K'_T$ such that
	\begin{equation} \label{EXmu-X1}
	\underset{0 \le t \le T}{\sup} \bE \left[  |\widehat{X}_t - X_t|^2 \right] \le  K_T \Delta,
	\end{equation}
and for any $p\in (0,2)$,
	\begin{equation} \label{EXmu-X0}
	 \bE \left[ \underset{0 \le t \le T}{\sup} |\widehat{X}_t - X_t|^p \right] \le  \left(\dfrac{4-p}{2-p}\right) (K'_T\Delta)^{p/2}.
	\end{equation}
	Moreover, when  $\widetilde{{\zeta_0}} < 0$ and $\alpha < 0$, there exists a positive constant  $K$, which does not depend on $T$,  such that
	\begin{equation} \label{EXmu-X2}
	\underset{t \ge 0}{\sup}\ \bE \left[  |\widehat{X}_t - X_t|^2 \right] \le  K \Delta.
	\end{equation} 
\end{Thm}

Note that the estimate  \eqref{T6} holds for  $\sigma_\Delta$ and $\gamma_{\Delta}$ given in \eqref{csmdelta1}.  

\begin{proof}	
		To simplify the exposition, we denote $Y_t:=X_t-\widehat{X}_t$. For any $\lambda\in\mathbb{R}$, applying  It\^o's formula to $e^{\lambda t}|Y_t|^2$ and using \eqref{eqn1} and \eqref{EM2}, we get
		\begin{align} 
		e^{\lambda t}|Y_t|^2 
		&= \mathcal{M}_t + \int_0^t  e^{\lambda s}\left[\lambda  |Y_s|^2 +2 \left\langle Y_s,  b(\theta_s,X_s)-b(\theta_{\underline{s}},\widehat{X}_{\underline{s}}) \right\rangle  + \left| \sm(\theta_s,X_s)-\sm_{\Delta}(\theta_{\underline{s}},\widehat{X}_{\underline{s}}) \right|^2 \right]ds  \notag\\ 
		&+ \int_0^t\int_{\mathbb{R}_0^d} e^{\lambda s}\left[ \left|\left(\gamma(\theta_s,X_{s})-\gamma_{\Delta}\left(\theta_{\underline{s}},\widehat{X}_{\underline{s}}\right)\right)z\right|^2    \right] \nu(d z) d s, 
		\label{tag25} 
		\end{align}
  where 
  \begin{align*}
      \mathcal{M}_t =  &2\int_0^t  e^{\lambda s}  \left\langle Y_s, \left(\sm(\theta_s,X_s)-\sm_{\Delta}(\theta_{\underline{s}},\widehat{X}_{\underline{s}})\right)dW_s \right\rangle\\
      &\quad + \int_{0}^{t} \int_{\mathbb{R}_0^d} e^{\lambda s}\left[\left|Y_{s-}+\left(\gamma\left(\theta_{s-},X_{s-}\right)-\gamma_{\Delta}\left(\theta_{\underline{s}-},\widehat{X}_{\underline{s}-}\right)\right) z\right|^2- \left|Y_{s-}\right|^2 \right] \widetilde{N}\left(d s, d z\right).
  \end{align*}
		
  Firstly, we have 
$$\left\langle Y_s, b\left(\theta_s,X_s\right)-b\left(\theta_{\underline{s}},\widehat{X}_{\underline{s}}\right) \right\rangle =\left\langle Y_s, b\left(\theta_s,X_s\right)-b\left(\theta_s,\widehat{X}_s\right) \right\rangle+\left\langle Y_s,  b\left(\theta_s,\widehat{X}_s\right) -b\left(\theta_{\underline{s}},\widehat{X}_{\underline{s}}\right) \right\rangle. $$
Using Cauchy's inequality, Condition \textbf{C1} and the inequality $(1+\vert x\vert+\vert y\vert)^2\leq 3\left(1+\vert x\vert^2+\vert y\vert^2\right)$ valid for all $x, y \in \bR^d$, we have that for any $\lambda_1>0$,
		\begin{align}
  &\left\langle Y_s,  b\left(\theta_s,\widehat{X}_s\right) -b\left(\theta_{\underline{s}},\widehat{X}_{\underline{s}}\right) \right\rangle \le  \lambda_1 |Y_s|^2+\dfrac{1}{4\lambda_1} \left| b\left(\theta_s,\widehat{X}_s\right) -b\left(\theta_{\underline{s}},\widehat{X}_{\underline{s}}\right)\right|^2 \notag \\
        & \leq  \lambda_1 |Y_s|^2+ \dfrac{1}{2\lambda_1} \Big( \Big|b\left(\theta_s, \Xm_s\right)-b\left(\theta_{s},\widehat{X}_{\underline{s}}\right)\Big|^2 + \Big|b\left(\theta_s, \Xm_{\underline{s}}\right)-b\left(\theta_{\underline{s}},\widehat{X}_{\underline{s}}\right)\Big|^2\Big) \notag \\
        & \leq  \lambda_1 |Y_s|^2 + \dfrac{1}{2\lambda_1}L_1^2\Big(1 + \left|\Xm_s\right|^l + \left|\widehat{X}_{\underline{s}}\right|^l\Big)^2 \left|\Xm_s - \widehat{X}_{\underline{s}}\right|^2 + \dfrac{1}{2\lambda_1}\Bigg(\Big|b\left(\theta_s, \Xm_{\underline{s}}\right)\Big| + \Big|b\left(\theta_{\underline{s}},\widehat{X}_{\underline{s}}\right)\Big|\Bigg)^2\mathbb{I}_{\{\theta_s\neq\theta_{\sd}\}} \notag \\
        & \leq  \lambda_1 |Y_s|^2 + \dfrac{3L_1^2}{2\lambda_1}\Big(1 + \left|\Xm_s\right|^{2l} + \left|\widehat{X}_{\underline{s}}\right|^{2l}\Big) \left|\Xm_s - \widehat{X}_{\underline{s}}\right|^2 + \dfrac{C}{2\lambda_1}\Big(1 + \left|\widehat{X}_{\underline{s}}\right|^{l+1}\Big)^2\mathbb{I}_{\{\theta_s\neq\theta_{\sd}\}} \notag \\
        &\leq  \lambda_1 |Y_s|^2 + \dfrac{3L_1^2}{2\lambda_1}\Big(1 + 2^{2l - 1}\left|\Xm_s - \Xm_{\underline{s}}\right|^{2l} + (2^{2l - 1} + 1)\left|{\Xm_{\underline{s}}}\right|^{2l}\Big) \left|\Xm_s - \widehat{X}_{\underline{s}}\right|^2 + \dfrac{C}{2\lambda_1}\Big(1 + \left|\widehat{X}_{\underline{s}}\right|^{l+1}\Big)^2\mathbb{I}_{\{\theta_s\neq\theta_{\sd}\}} \notag \\
        & \leq  \dfrac{3L_1^2}{2\lambda_1}\left|\Xm_s - \widehat{X}_{\underline{s}}\right|^2 + \dfrac{3L_1^2}{\lambda_1}2^{2l - 2}\left|\Xm_s - {\Xm_{\underline{s}}}\right|^{2l+2} \notag \\
        &\quad + \dfrac{3L_1^2}{2\lambda_1} \left(2^{2l - 1} + 1\right)\left|{\Xm_{\underline{s}}}\right|^{2l}\left|\Xm_s - \widehat{X}_{\underline{s}}\right|^2 + \dfrac{C}{\lambda_1}\Big(1 + \left|\widehat{X}_{\underline{s}}\right|^{2l+2}\Big)\mathbb{I}_{\{\theta_s\neq\theta_{\sd}\}} + \lambda_1 |Y_s|^2,\label{tag14} 
		\end{align}
for some positive constant $C = C(L_1, l, \max_{i\in S}\{|b(i,0)|\})$.\\
Secondly, for any $\lambda_2>0$, it holds that  
$$|a-b|^2 \leq (1+\lambda_2)|a-c|^2 + 2 \left(1+ \frac{1}{\lambda_2}\right)(|c-d|^2+ |d-b|^2), \quad a, b, c, d \in \mathbb{R}^d, \lambda_2 > 0.$$
Thus, 
\begin{align}
	\left|\sigma\left(\theta_s, X_s\right)-\sigma_\Delta\left(\theta_{\underline{s}},\widehat{X}_{\underline{s}}\right)\right|^2 
	&\le  \left(1+\lambda_2\right) \left|\sigma\left(\theta_s, X_s\right) - \sigma\left(\theta_s, \Xm_s\right)\right|^2 +   2\left(1+\dfrac{1}{\lambda_2}\right) \left|\sigma\left(\theta_s, \widehat{X}_s\right)-\sigma\left(\theta_s, \widehat{X}_{\underline{s}}\right)\right|^2 \notag \\
    &\quad + 2\left(1+\dfrac{1}{\lambda_2}\right)\left|\sigma\left(\theta_s, \widehat{X}_{\underline{s}}\right)-\sigma_\Delta\left(\theta_{\underline{s}}, \widehat{X}_{\underline{s}}\right)\right|^2.  \notag
\end{align}
Thanks to Condition \textbf{C2} and \eqref{T6}, we have that for any $\lambda_2>0$,
\begin{align*}
& \left|\sigma\left(\theta_s, \widehat{X}_s\right)-\sigma\left(\theta_s, \widehat{X}_{\underline{s}}\right)\right|^2  + \left|\sigma\left(\theta_s, \widehat{X}_{\underline{s}}\right)-\sigma_\Delta\left(\theta_{\underline{s}}, \widehat{X}_{\underline{s}}\right)\right|^2 \\
    &\le  L_2^2\left(1+|\widehat{X}_s|^{m}+|\widehat{X}_{\underline{s}}|^{m}\right)^2 \left| \widehat{X}_s - \widehat{X}_{\underline{s}}\right|^2  +2\left|\sigma\left(\theta_s, \widehat{X}_{\underline{s}}\right)-\sigma\left(\theta_{\underline{s}}, \widehat{X}_{\underline{s}}\right)\right|^2 + 2 \left|\sigma\left(\theta_{\underline{s}}, \widehat{X}_{\underline{s}}\right)-\sigma_\Delta\left(\theta_{\underline{s}}, \widehat{X}_{\underline{s}}\right)\right|^2 \notag \\
    &\le 3L_2^2 \left(1+|\widehat{X}_s|^{2m}+|\widehat{X}_{\underline{s}}|^{2m}\right) \left| \widehat{X}_s - \widehat{X}_{\underline{s}}\right|^2  + 2\left(\left|\sigma\left(\theta_s, \widehat{X}_{\underline{s}}\right)\right|+\left|\sigma\left(\theta_{\underline{s}}, \widehat{X}_{\underline{s}}\right)\right|\right)^2\mathbb{I}_{\{\theta_s\neq\theta_{\sd}\}} + 2L_4^2 \Delta\left| \sm\left(\theta_{\underline{s}},\widehat{X}_{\underline{s}}\right)\right|^4 \notag \\
	&\le  3 L_2^2\left(1+2^{2m-1}\left|  \widehat{X}_s - \widehat{X}_{\underline{s}} \right|^{2m} + \left(2^{2m-1}+1\right)|\widehat{X}_{\underline{s}}|^{2m}\right) \left| \widehat{X}_s - \widehat{X}_{\underline{s}}  \right|^2 \notag\\
	&\quad + C \left(1+|\widehat{X}_{\underline{s}}|^{2m+2}\right)\mathbb{I}_{\{\theta_s\neq\theta_{\sd}\}} + CL_4^2 \Delta \left(1+|\widehat{X}_{\underline{s}}|^{4m+4}\right) \notag \\
	&=     3 L_2^2 \left| \widehat{X}_s  -   \widehat{X}_{\underline{s}}  \right|^2+ 3 L_2^22^{2m-1}\left|  \widehat{X}_s - \widehat{X}_{\underline{s}} \right|^{2m+2} + 3L_2^2\left(2^{2m-1}+1\right)|\widehat{X}_{\underline{s}}|^{2m} \left| \widehat{X}_s  -   \widehat{X}_{\underline{s}}  \right|^2 \notag \\
 &\quad + C\left(1+|\widehat{X}_{\underline{s}}|^{2m+2}\right)\mathbb{I}_{\{\theta_s\neq\theta_{\sd}\}} + C\left(1+\dfrac{1}{\lambda_2}\right)L_4^2 \Delta \left(1+|\widehat{X}_{\underline{s}}|^{4m+4} \right), \label{tag15}
\end{align*}
for some positive constant $C=C\left(L_2,m,\max_{i\in S}\{|\sm(i,0)|\}\right)$.

Third,  for any $\lambda_3>0$, we have 
\begin{align}
     \left|\gamma(\theta_s,X_{s})-\gamma_{\Delta}\left(\theta_{\underline{s}},\widehat{X}_{\underline{s}}\right) \right|^2 \le & \left| \gamma(\theta_s,X_s)-\gamma(\theta_s,\widehat{X}_s)\right|^2+ \left| \gamma(\theta_s,\widehat{X}_s) -  \gamma_{\Delta}(\theta_{\underline{s}},\widehat{X}_{\underline{s}}) \right|^2 \notag \\ 
     &+ \lambda_3 \left| \gamma(\theta_s,X_s)-\gamma(\theta_s,\widehat{X}_s)\right|^2 + \dfrac{1}{\lambda_3}\left| \gamma(\theta_s,\widehat{X}_s) -  \gamma_{\Delta}(\theta_{\underline{s}},\widehat{X}_{\underline{s}}) \right|^2. \notag
\end{align}
Using Conditions \textbf{C1}, \textbf{C3} and \eqref{T6}, we have 
\begin{align}
    &\left|\gamma(\theta_s,X_{s})-\gamma_{\Delta}\left(\theta_{\underline{s}},\widehat{X}_{\underline{s}}\right) \right|^2  \\
    &\le \left| \gamma(\theta_s,X_s)-\gamma(\theta_s,\widehat{X}_s)\right|^2 + \lambda_3 L_3^2 \left|  X_s - \widehat{X}_s \right|^2 + 2\left(1+\dfrac{1}{\lambda_3}\right)L_3^2 \left| \widehat{X}_s  -   \widehat{X}_{\underline{s}}  \right|^2 \notag \\
    &\quad+ 4\left(1+\dfrac{1}{\lambda_3}\right)\left| \gamma(\theta_s,\widehat{X}_{\underline{s}}) -  \gamma(\theta_{\underline{s}},\widehat{X}_{\underline{s}}) \right|^2 + 4\left(1+\dfrac{1}{\lambda_3}\right)\left| \gamma(\theta_{\underline{s}},\widehat{X}_{\underline{s}}) -  \gamma_{\Delta}(\theta_{\underline{s}},\widehat{X}_{\underline{s}}) \right|^2  \notag \\
    &\le \left| \gamma(\theta_s,X_s)-\gamma(\theta_s,\widehat{X}_s)\right|^2 + \lambda_3 L_3^2 \left|  X_s - \widehat{X}_s \right|^2 + 2\left(1+\dfrac{1}{\lambda_3}\right)L_3^2 \left| \widehat{X}_s  -   \widehat{X}_{\underline{s}}  \right|^2 \notag \\
    &\quad+ C\left(1+\dfrac{1}{\lambda_3}\right)\left(1 + \left|\widehat{X}_{\underline{s}}\right|^2\right)\mathbb{I}_{\{\theta_s\neq\theta_{\sd}\}} + C\left(1+\dfrac{1}{\lambda_3}\right)L_4^2 \Delta  \left(1+\left|\widehat{X}_{\underline{s}}\right|^{2l+6}\right),
	\label{tag36}
\end{align}
for some positive constant $C=C(L_1,L_3,l,\max_{i\in S}\{|b(i,0)|\},\max_{i\in S}\{|\gamma(i,0)|\})$.

Therefore, inserting \eqref{tag14}, \eqref{tag15} and \eqref{tag36} into \eqref{tag25}, choosing $\lambda_2=\epsilon$ sufficiently small, and using Condition \textbf{C6} and $Y_s=X_s-\widehat{X}_s$, we obtain that for any $\lambda_1, \lambda_3>0$,
\begin{align} 
	&e^{\lambda t}|Y_t|^2 
    \le  \mathcal{M}_t + \int_0^t  e^{\lambda  s}\Bigg[\lambda  |Y_s|^2 + 2\left\langle Y_s, b\left(\theta_s, X_s\right)-b\left(\theta_{s},\widehat{X}_s\right)\right\rangle + \dfrac{3L_1^2}{\lambda_1}\left|\Xm_s - \widehat{X}_{\underline{s}}\right|^2 + \dfrac{6L_1^2}{\lambda_1}2^{2l - 2}\left|\Xm_s - {\Xm_{\underline{s}}}\right|^{2l+2} \notag \\
    &\quad + \dfrac{3L_1^2}{\lambda_1} \left(2^{2l - 1} + 1\right)\left|{\Xm_{\underline{s}}}\right|^{2l}\left|\Xm_s - \widehat{X}_{\underline{s}}\right|^2 + \dfrac{C}{\lambda_1}\Big(1 + \left|\widehat{X}_{\underline{s}}\right|^{2l+2}\Big)\mathbb{I}_{\{\theta_s\neq\theta_{\sd}\}} + 2\lambda_1 |Y_s|^2  \notag \\
    &+ \left(1 + \epsilon\right)\left| \sm(\theta_s,X_s)-\sm(\theta_s,\widehat{X}_s) \right|^2 +6\left(1+\dfrac{1}{\epsilon}\right)L_2^2 \left| \widehat{X}_s  -   \widehat{X}_{\underline{s}}  \right|^2+ 6\left(1+\dfrac{1}{\epsilon}\right)L_2^22^{2m-1}\left|  \widehat{X}_s - \widehat{X}_{\underline{s}} \right|^{2m+2}\notag\\ 
    &\quad+6\left(1+\dfrac{1}{\epsilon}\right)L_2^2\left(2^{2m-1}+1\right)|\widehat{X}_{\underline{s}}|^{2m} \left| \widehat{X}_s  -   \widehat{X}_{\underline{s}}  \right|^2 + C\left(1+\dfrac{1}{\epsilon}\right)\left(1+|\widehat{X}_{\underline{s}}|^{2m+2}\right)\mathbb{I}_{\{\theta_s\neq\theta_{\sd}\}} \notag \\
    &\quad + C\left(1+\dfrac{1}{\epsilon}\right)L_4^2 \Delta \left(1+|\widehat{X}_{\underline{s}}|^{4m+4} \right)\Bigg]ds  \notag \\
	&\quad + \int_0^t\int_{\mathbb{R}_0^d} e^{\lambda s}\Bigg[ \left| \gamma(\theta_s,X_s)-\gamma(\theta_s,\widehat{X}_s)\right|^2 + \lambda_3 L_3^2 \left|  X_s - \widehat{X}_s \right|^2 + 2\left(1+\dfrac{1}{\lambda_3}\right)L_3^2 \left| \widehat{X}_s  -   \widehat{X}_{\underline{s}}  \right|^2 \notag \\
    &\quad+ C\left(1+\dfrac{1}{\lambda_3}\right)\left(1 + \left|\widehat{X}_{\underline{s}}\right|^2\right)\mathbb{I}_{\{\theta_s\neq\theta_{\sd}\}} + C\left(1+\dfrac{1}{\lambda_3}\right)L_4^2 \Delta  \left(1+\left|\widehat{X}_{\underline{s}}\right|^{2l+6}\right)\Bigg] |z|^2 \nu(d z) d s \notag \\
    &=  \mathcal{M}_t + \left(\lambda + 2\lambda_1 + \lambda_3 L_3^2\int_{\mathbb{R}_0^d}|z|^2\nu(d z)\right)\int_0^t e^{\lambda  s}|Y_s|^2 ds \notag \\
    &\quad+ \int_0^t e^{\lambda  s}\left[2\left\langle Y_s, b\left(\theta_s, X_s\right)-b\left(\theta_{s},\widehat{X}_s\right)\right\rangle + \left(1 + \epsilon\right)\left| \sm(\theta_s,X_s)-\sm(\theta_s,\widehat{X}_s) \right|^2 \right] ds \notag \\
    &\quad+ \int_0^t e^{\lambda  s} \left| \gamma(\theta_s,X_s)-\gamma(\theta_s,\widehat{X}_s)\right|^2\int_{\mathbb{R}_0^d}|z|^2\nu(d z) ds \notag \\
    &\quad+ \int_0^t e^{\lambda  s}\Bigg[\left(\dfrac{3L_1^2}{\lambda_1} + 6\left(1+\dfrac{1}{\epsilon}\right)L_2^2 + 2\left(1+\dfrac{1}{\lambda_3}\right)L_3^2\int_{\mathbb{R}_0^d}|z|^2\nu(d z)\right)\left|\Xm_s - \widehat{X}_{\underline{s}}\right|^2 + \dfrac{6L_1^2}{\lambda_1}2^{2l - 2}\left|\Xm_s - {\Xm_{\underline{s}}}\right|^{2l+2} \notag \\ 
    &\quad+ \dfrac{3L_1^2}{\lambda_1} \left(2^{2l - 1} + 1\right)\left|{\Xm_{\underline{s}}}\right|^{2l}\left|\Xm_s - \widehat{X}_{\underline{s}}\right|^2 + \dfrac{C}{\lambda_1}\Big(1 + \left|\widehat{X}_{\underline{s}}\right|^{2l+2}\Big)\mathbb{I}_{\{\theta_s\neq\theta_{\sd}\}} + 6\left(1+\dfrac{1}{\epsilon}\right)L_2^22^{2m-1}\left|  \widehat{X}_s - \widehat{X}_{\underline{s}} \right|^{2m+2}\notag\\ 
    &\quad+6\left(1+\dfrac{1}{\epsilon}\right)L_2^2\left(2^{2m-1}+1\right)|\widehat{X}_{\underline{s}}|^{2m} \left| \widehat{X}_s  -   \widehat{X}_{\underline{s}}  \right|^2 + C\left(1+\dfrac{1}{\epsilon}\right)\left(1+|\widehat{X}_{\underline{s}}|^{2m+2}\right)\mathbb{I}_{\{\theta_s\neq\theta_{\sd}\}} \notag \\
    &\quad + C\left(1+\dfrac{1}{\epsilon}\right)L_4^2 \Delta \left(1+|\widehat{X}_{\underline{s}}|^{4m+4} \right) + \Bigg(C\left(1+\dfrac{1}{\lambda_3}\right)\left(1 + \left|\widehat{X}_{\underline{s}}\right|^2\right)\mathbb{I}_{\{\theta_s\neq\theta_{\sd}\}} \notag \\ 
    &\quad + C\left(1+\dfrac{1}{\lambda_3}\right)L_4^2 \Delta  \left(1+\left|\widehat{X}_{\underline{s}}\right|^{2l+6}\right)\Bigg)\int_{\mathbb{R}_0^d}|z|^2 \nu(d z) \Bigg] ds \notag \\
    &\le \mathcal{M}_t +  \left(\alpha + \lambda + 2\lambda_1 + \lambda_3 L_3^2\int_{\mathbb{R}_0^d}|z|^2\nu(d z)\right)\int_0^t e^{\lambda  s}|Y_s|^2 ds \notag \\
    &\quad+ \int_0^t e^{\lambda  s}\Bigg[\left(\dfrac{3L_1^2}{\lambda_1} + 6\left(1+\dfrac{1}{\epsilon}\right)L_2^2 + 2\left(1+\dfrac{1}{\lambda_3}\right)L_3^2\int_{\mathbb{R}_0^d}|z|^2\nu(d z)\right)\left|\Xm_s - \widehat{X}_{\underline{s}}\right|^2 \notag \\
     &\quad+  \dfrac{6L_1^2}{\lambda_1}2^{2l - 2}\left|\Xm_s - {\Xm_{\underline{s}}}\right|^{2l+2} + \dfrac{3L^2_1}{\lambda_1} \left(2^{2l - 1} + 1\right)\left|{\Xm_{\underline{s}}}\right|^{2l}\left|\Xm_s - \widehat{X}_{\underline{s}}\right|^2 \notag \\ 
    &\quad+ \dfrac{C}{\lambda_1}\Big(1 + \left|\widehat{X}_{\underline{s}}\right|^{2l+2}\Big)\mathbb{I}_{\{\theta_s\neq\theta_{\sd}\}} + 6\left(1+\dfrac{1}{\epsilon}\right)L_2^22^{2m-1}\left|  \widehat{X}_s - \widehat{X}_{\underline{s}} \right|^{2m+2}\notag\\ 
    &\quad+6\left(1+\dfrac{1}{\epsilon}\right)L_2^2\left(2^{2m-1}+1\right)|\widehat{X}_{\underline{s}}|^{2m} \left| \widehat{X}_s  -   \widehat{X}_{\underline{s}}  \right|^2 + C\left(1+\dfrac{1}{\epsilon}\right)\left(1+|\widehat{X}_{\underline{s}}|^{2m+2}\right)\mathbb{I}_{\{\theta_s\neq\theta_{\sd}\}} \notag \\
    &\quad + C\left(1+\dfrac{1}{\epsilon}\right)L_4^2 \Delta \left(1+|\widehat{X}_{\underline{s}}|^{4m+4} \right) + \Bigg(C\left(1+\dfrac{1}{\lambda_3}\right)\left(1 + \left|\widehat{X}_{\underline{s}}\right|^2\right)\mathbb{I}_{\{\theta_s\neq\theta_{\sd}\}} \notag \\ 
    &\quad + C\left(1+\dfrac{1}{\lambda_3}\right)L_4^2 \Delta  \left(1+\left|\widehat{X}_{\underline{s}}\right|^{2l+6}\right)\Bigg)\int_{\mathbb{R}_0^d}|z|^2 \nu(d z) \Bigg] ds 
    \label{estimate}
\end{align}
for some positive constant $C = C(L_1,L_2,L_3,l,m,\max_{i\in S}\{|b(i,0)|\},\max_{i\in S}\{|\gamma(i,0)|\},\max_{i\in S}\{|\sm(i,0)|\})$.\\
Then, taking $\lambda = \lambda_0 := - \left(\alpha  + 2\lambda_1+\lambda_3 L_3^2 \int_{\bR_{0}^{d}}|z|^2 \nu(dz)\right)$ in \eqref{estimate} yields to
\begin{align}
	e^{\lambda_0 t}|Y_t|^2 
	&\le \mathcal{M}_t + \int_0^t e^{\lambda_0  s}\Bigg[\left(\dfrac{3L_1^2}{\lambda_1} + 6\left(1+\dfrac{1}{\epsilon}\right)L_2^2 + 2\left(1+\dfrac{1}{\lambda_3}\right)L_3^2\int_{\mathbb{R}_0^d}|z|^2\nu(d z)\right)\left|\Xm_s - \widehat{X}_{\underline{s}}\right|^2 \notag \\
 &\quad + \dfrac{6L_1^2}{\lambda_1}2^{2l - 2}\left|\Xm_s - {\Xm_{\underline{s}}}\right|^{2l+2} + \dfrac{3L_1^2}{\lambda_1} \left(2^{2l - 1} + 1\right)\left|{\Xm_{\underline{s}}}\right|^{2l}\left|\Xm_s - \widehat{X}_{\underline{s}}\right|^2\notag \\ 
    &\quad + \dfrac{C}{\lambda_1}\Big(1 + \left|\widehat{X}_{\underline{s}}\right|^{2l+2}\Big)\mathbb{I}_{\{\theta_s\neq\theta_{\sd}\}} + 6\left(1+\dfrac{1}{\epsilon}\right)L_2^22^{2m-1}\left|  \widehat{X}_s - \widehat{X}_{\underline{s}} \right|^{2m+2}\notag\\ 
    &\quad+6\left(1+\dfrac{1}{\epsilon}\right)L_2^2\left(2^{2m-1}+1\right)|\widehat{X}_{\underline{s}}|^{2m} \left| \widehat{X}_s  -   \widehat{X}_{\underline{s}}  \right|^2 + C\left(1+\dfrac{1}{\epsilon}\right)\left(1+|\widehat{X}_{\underline{s}}|^{2m+2}\right)\mathbb{I}_{\{\theta_s\neq\theta_{\sd}\}} \notag \\
    &\quad + C\left(1+\dfrac{1}{\epsilon}\right)L_4^2 \Delta \left(1+|\widehat{X}_{\underline{s}}|^{4m+4} \right) + \Bigg(C\left(1+\dfrac{1}{\lambda_3}\right)\left(1 + \left|\widehat{X}_{\underline{s}}\right|^2\right)\mathbb{I}_{\{\theta_s\neq\theta_{\sd}\}} \notag \\ 
    &\quad + C\left(1+\dfrac{1}{\lambda_3}\right)L_4^2 \Delta  \left(1+\left|\widehat{X}_{\underline{s}}\right|^{2l+6}\right)\Bigg)\int_{\mathbb{R}_0^d}|z|^2 \nu(d z) \Bigg] ds. 
    \label{estimate1}
\end{align}
Now, using Lemma \ref{hq2}, Proposition \ref{dinh ly 2} and the fact that $p_0 \ge \max\{2l+6, 4m+4\}$, we get
\begin{align}
	&\bE \left[\left|\widehat{X}_s-\widehat{X}_{\underline{s}}\right|^{2}\right] \leq C\Delta; \quad \bE \left[\left|\widehat{X}_s-\widehat{X}_{\underline{s}}\right|^{2l+2}\right] \leq C\Delta; \quad \bE \left[\left|\widehat{X}_s-\widehat{X}_{\underline{s}}\right|^{2m+2}\right] \leq C\Delta, \notag \\
	&\bE \left[|\widehat{X}_{\underline{s}}|^{4m+4}\right]\leq C; \quad \bE \left[|\widehat{X}_{\underline{s}}|^{2l+6}\right]\leq C, \label{estimate4}
\end{align}
and for $\upsilon \in\{l,m\}$,
\begin{align}
	\bE \left[|\widehat{X}_{\underline{s}}|^{2\upsilon} |\widehat{X}_s-\widehat{X}_{\underline{s}}|^{2}\right] =\bE \left[\bE \left[|\widehat{X}_{\underline{s}}|^{2\upsilon} |\widehat{X}_s-\widehat{X}_{\underline{s}}|^{2}\vert  \mathcal{F}_{\underline{s}}\right] \right] 
    &=\bE \left[|\widehat{X}_{\underline{s}}|^{2\upsilon}\bE \left[ |\widehat{X}_s-\widehat{X}_{\underline{s}}|^{2}\vert  \mathcal{F}_{\underline{s}}\right] \right]  \notag 	\\
	& \leq C\Delta \bE \left[|\widehat{X}_{\underline{s}}|^{2\upsilon}\right]  \leq C\Delta, \label{estimate2}
\end{align}
for some positive constant $C$.
Note that 
\begin{equation*}
    \begin{split}
    \displaystyle \mathbb{E}\left[\mathbb{I}_{\{\theta_s\neq\theta_{\sd}\}}|\mathcal{F}_{\sd}\right] 
    & = \displaystyle\sum_{i\in S}\mathbb{I}_{\{\theta_{\sd}=i\}}\sum_{j \in S;j\neq i}\pr\left[\theta_s = j|\theta_{\sd}=i\right] = \displaystyle\sum_{i\in S}\mathbb{I}_{\{\theta_{\sd}=i\}}\sum_{j \in S;j\neq i}\left(\vartheta_{ij}(s - \sd) + o(s-\sd)\right)\\
    & \leq \displaystyle C\left(\max_{i\in S}(-\vartheta_{ii})(s-\sd) + o(s-\sd)\right)\sum_{i\in S}\mathbb{I}_{\{\theta_{\sd}=i\}} \leq C\Delta.
    \end{split}
\end{equation*}
    Thus, 
    \begin{equation} \label{key}
    \mathbb{E}\left[\left(1+|\Xm_{\sd}|^{2g+2}\right)\mathbb{I}_{\{\theta_s\neq\theta_{\sd}\}}\right]\leq C\Delta,
    \end{equation}
for some positive constant $C$ for any $g \in \{l,m,0\}$.

Consequently, plugging \eqref{estimate4},\eqref{estimate2}, \ref{key} into \eqref{estimate1}, taking the expectation on both sides, we get that for any $t \in[0,T]$,
		\begin{align*}
		\bE \left[e^{\lambda_0 t}|Y_t|^2 \right]
		&\le  C_T \Delta \int_{0}^{t} e^{\lambda_0 s}ds. 
		\end{align*}	
for some positive constant $C_T$. This implies \eqref{EXmu-X1}. 
Next, let $\tau $ be any stopping time bounded by $T$.  It follows from \eqref{estimate1} and the estimates \eqref{estimate4},\eqref{estimate2}, \eqref{key} that 
	\begin{align*}
	&\bE \left[e^{\lambda_0 \tau}|Y_{\tau}|^2 \right]\\
	&\le \int_0^T e^{\lambda_0  s}\bE\Bigg[\left(\dfrac{3L_1^2}{\lambda_1} + 6\left(1+\dfrac{1}{\epsilon}\right)L_2^2 + 2\left(1+\dfrac{1}{\lambda_3}\right)L_3^2\int_{\mathbb{R}_0^d}|z|^2\nu(d z)\right)\left|\Xm_s - \widehat{X}_{\underline{s}}\right|^2 + \dfrac{6L_1^2}{\lambda_1}2^{2l - 2}\left|\Xm_s - {\Xm_{\underline{s}}}\right|^{2l+2} \notag \\ 
    &\quad+ \dfrac{3L_1^2}{\lambda_1} \left(2^{2l - 1} + 1\right)\left|{\Xm_{\underline{s}}}\right|^{2l}\left|\Xm_s - \widehat{X}_{\underline{s}}\right|^2 + \dfrac{C}{\lambda_1}\Big(1 + \left|\widehat{X}_{\underline{s}}\right|^{2l+2}\Big)\mathbb{I}_{\{\theta_s\neq\theta_{\sd}\}} + 6\left(1+\dfrac{1}{\epsilon}\right)L_2^22^{2m-1}\left|  \widehat{X}_s - \widehat{X}_{\underline{s}} \right|^{2m+2}\notag\\ 
    &\quad+6\left(1+\dfrac{1}{\epsilon}\right)L_2^2\left(2^{2m-1}+1\right)|\widehat{X}_{\underline{s}}|^{2m} \left| \widehat{X}_s  -   \widehat{X}_{\underline{s}}  \right|^2 + C\left(1+\dfrac{1}{\epsilon}\right)\left(1+|\widehat{X}_{\underline{s}}|^{2m+2}\right)\mathbb{I}_{\{\theta_s\neq\theta_{\sd}\}} \notag \\
    &\quad + C\left(1+\dfrac{1}{\epsilon}\right)L_4^2 \Delta \left(1+|\widehat{X}_{\underline{s}}|^{4m+4} \right) + \Bigg(C\left(1+\dfrac{1}{\lambda_3}\right)\left(1 + \left|\widehat{X}_{\underline{s}}\right|^2\right)\mathbb{I}_{\{\theta_s\neq\theta_{\sd}\}} \notag \\ 
    &\quad + C\left(1+\dfrac{1}{\lambda_3}\right)L_4^2 \Delta  \left(1+\left|\widehat{X}_{\underline{s}}\right|^{2l+6}\right)\Bigg)\int_{\mathbb{R}_0^d}|z|^2 \nu(d z) \Bigg] ds \le  \widetilde{C}_T \Delta,
	\end{align*}
for some constant $\widetilde{C}_T>0$ not depending on the stopping time $\tau$. Applying Proposition IV.4.7 in \cite{RY}, we obtain for any $p \in (0,2)$ that 
	\begin{align*}
	\bE \left[ \underset{0 \le t \le T}{\sup} e^{\frac{p \lambda_0 t}{2}}|Y_t|^p \right] \leq \left(\dfrac{2-p/2}{1-p/2}\right)  (\widetilde{C}_T\Delta)^{p/2}.
	\end{align*}
Since  $e^{\frac{p \lambda_0 t}{2}}\geq e^{-\frac{p \vert\lambda_0 \vert T}{2}}$,  we concludes \eqref{EXmu-X0}. 

When $\alpha<0$, we can always choose $\lambda_1, \lambda_3>0$ such that
$
\alpha  + 2\lambda_1+\lambda_3 L_3^2 \int_{\bR_{0}^{d}}|z|^2 \nu(dz)<0.
$
From this, the chosen $\lambda_0$ is negative. As a c1111onsequence,  when $\alpha<0$ and $\widetilde{{\zeta_0}}  <0$, the constant $C_T$ can be chosen in a way that does not depend on $T$. Therefore, we obtain \eqref{EXmu-X2}. This finishes the desired proof.
\end{proof} 

From Theorem \ref{dinh ly 6}, the approximation scheme (\ref{EM1}) gives an approximated solution for SDE (\ref{eqn1}) that satisfies \eqref{MLMC1}. Moreover, in the next Theorem, one can see that if the length of estimation interval $T$, the number of levels $L$ and the number of samples $N_l$ for all levels $l$ are chosen as in Theorem \ref{MSEbound}, the expected overall estimation cost will be $O(\varepsilon^{-2}|\log \varepsilon|)$.

\begin{Thm}
    If the length of estimation interval $T$, the number of levels $L$ and the number of samples $N_l$ for all levels $l$ are chosen as in Theorem \ref{MSEbound}, then there exists a positive constant $C_3$ such that
    $$
    \textbf{C}_{MLMC} \leq C_3\varepsilon^{-2}|\log \varepsilon|^3.
    $$
\end{Thm}\label{TheoryComputationalCost}
\begin{proof}
    We have that the expected overall cost of the Multi-level estimator is $\sum_{l=0}^L N_l \overline{C}_l$. Thus, we can have that
    \begin{align*}
        \textbf{C}_{MLMC} 
         = \sum_{l=0}^L N_l T M^l 
        &= \sum_{l=0}^L \left\lceil3 \varepsilon^{-2}M^{-l}C_2 \left( \left \lceil \dfrac{2|\log \varepsilon|}{\log M} + \dfrac{\log(6\mathbf{K})_0 L_\varphi^2}{\log M} \right \rceil +1 \right)\right\rceil T M^l \leq C_3 \varepsilon^{-2}|\log \varepsilon|^3,
    \end{align*}
    for some constant $C_3$ depending on $C_2, M, \kappa_1$ and $\alpha$.
\end{proof}

\section{Numerical experiments} \label{sec:nume}
\subsection{Path approximation}
We consider numerical experiments for Markovian switching SDE  with jumps \eqref{eqn1} with $N = 2$, $d= 3$, $X_0 = 0 \in \mathbb{R}^3$, 
\begin{align*}
b(1, x) &= (1- x_1 - x_1^3; 1 - x_2 - x_2^3; 1 - x_3 - x_3^3)^\mathsf{T}, \quad
b(2, x) = (2- x_1 - x_1^3; 2 - x_2 - x_2^3; 2 - x_3 - x_3^3)^\mathsf{T}, \\
 \sigma(1, x) &= 0.3 \times  \text{diag} (x_2^2; x_3^2; x_1^2), \quad
 \sigma(2, x) = 0.3 \times  \text{diag} (x_1^2; x_2^2; x_3^2), \\
  \gamma(1, x) &= 0.2 \times \text{diag}(x_1; x_2 + \sin x_3; \cos x_3), \quad
 \gamma(2, x) = 0.2 \times \text{diag}(\cos x_3; x_1; x_2 + \sin x_3).
 \end{align*}
It can be verified that  such coefficients satisfy Conditions \textbf{C1}--\textbf{C6},  \eqref{T5}, and  \eqref{T6} with $l = 2, m = 1, p_0= 10$, ${\zeta_0} = \tilde{{\zeta_0}}< 0$, and $\alpha <0$. 
 Therefore, it follows from Theorem \ref{dinh ly 6} that the TAEM approximation scheme {defined} by \eqref{EM1}, \eqref{chooseh} and \eqref{csmdelta1} converges  with the rate $1/2$ in $L^2$-norm on any finite and infinite time intervals. Let  $Z= (Z^1, Z^2, Z^3)^\mathsf{T}$. We consider the following states of activity of $Z$.

\begin{itemize}
    \item \emph{Finite activity}: For each $j = 1, 2, 3$, let $Z^j$ be a compound Poisson process of the form  $Z^j_t = \sum_{i=1}^{P^j_t} \xi^j_i,$ where $(P^j_t)_{t \geq 0}, j = 1, 2, 3,$ are three independent Poisson processes with the same intensity $\lambda = 10$, and  the jump amplitude $(\xi^j_i)_{i,j\geq 1}$ is an array of independent and $\mathcal{N}(0, 0.4^2)$-distributed random variables.
    \item \emph{Infinite activity}: 
    Let $Z = (Z_t)_{t \geq 0}$ be a bilateral Gamma process with  scale parameter
     $10$ and shape parameter $1$.
\end{itemize}

For each activity state, the function $\textrm{MSE}(\ell, T)$ represents the mean squared error between levels $\ell$ and $\ell+1$ for the approximation $\widehat{X}$, defined as
$$\textrm{MSE}(\ell, T) = \frac{1}{M} \sum_{k=1}^M |\widehat{X}^{(\ell,k)}_T - \widehat{X}^{(\ell+1,k)}_T|^2,$$
where $\widehat{X}^{(\ell,1)}, \ldots, \widehat{X}^{(\ell,M)}$ are $M$ independent copies of $\widehat{X}^{(\ell)}$, constructed according to equations \eqref{EM1}, \eqref{chooseh}, and \eqref{csmdelta1} with step size $\Delta = 2^{-\ell}$ for each level $\ell \geq 1$. Additionally, it is essential that $\widehat{X}^{(\ell,k)}_T$ and $\widehat{X}^{(\ell+1,k)}_T$ be simulated using the same Brownian motion and Lévy process sample paths (see Algorithm 1 in \cite{FG}).

It can be seen that if there exists some positive number $\Lambda_0$ such that $2^{\Lambda_0 \ell} \|\widehat{X}^{(\ell+1)}_T - \widehat{X}^{(\ell)}_T\|_{L^2} = O(1)$, then this is equivalent to the fact that $\widehat{X}^{(\ell)}$ possesses some convergence rate of order $\Lambda_0 \in (0,+\infty)$ in $L^2$-norm. It follows that $\log_2 \textrm{MSE}(\ell,T) = -2\Lambda_0 \ell + \Lambda_1 + o(1),  $
for some constant $\Lambda_1$. This suggests using a linear regression model to estimate the empirical convergence rate $\Lambda_0$.

Table \ref{table1} and Table \ref{table2} show the values of $\log_2 \textrm{MSE}(\ell,T)$ with $T=5$ and $T=10$ and $\ell = 1,\ldots, 6$. The maximum likelihood estimator of $\Lambda_0$ are  $0.6$ in the finite activity case and $0.55$ in the infinite activity case. 
Moreover, it can be seen that the values of $\log_2 \textrm{MSE}(\ell,5)$ and $\log_2 \textrm{MSE}(\ell,10)$ are almost the same. These facts  support our theoretical result stated in Theorem \ref{dinh ly 6}. 

\begin{table}
\begin{center}
\begin{tabular}{c|ccccc}
$\ell$&	2&	3&	4&	5&	6\\
\hline 
$T=5$	& -9.21	& -11.06	&-12.55	&-13.59	&-14.67\\
\hline 
$T=10$	&-9.16	&-11.02	&-12.14	&-13.48	&-14.68
\end{tabular}
\caption{Finite activity: The values of error $\log_2 \textrm{MSE}(\ell,T)$ with $T=5$ and $T=10$.}
\label{table1}
\end{center}
\end{table}

\begin{table}
\begin{center}
\begin{tabular}{c|ccccc}
$\ell$&	2&	3&	4&	5&	6\\
\hline 
$T=5$	& -8.29	& -10.04	&-11.34	&-12.56	&-13.56\\
\hline 
$T=10$	&-8.42	&-10.12	&-11.11	&-12.59	&-13.22
\end{tabular}
\caption{Infinite activity: The values of error $\log_2 \textrm{MSE}(\ell,T)$ with $T=5$ and $T=10$.}
\label{table2}
\end{center}
\end{table}

\subsection{Invariant measure approximation}
In this subsection, we show our simulation result for the MLMC approximation of the invariant measure, whose theory was presented in Section \ref{InvaMeasAppro}.

\begin{figure}
\begin{subfigure}{.5\textwidth}
    \centering
    \includegraphics[height = 5cm]{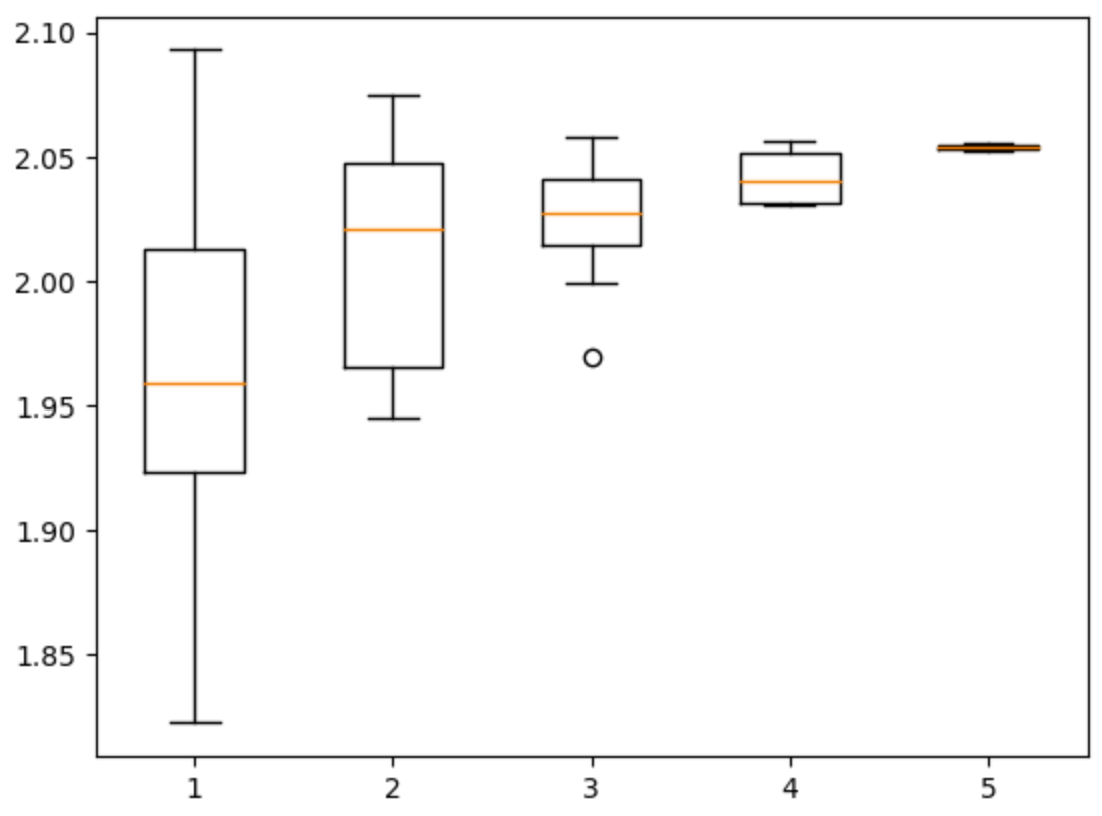}
    \caption{Relation between $\log_2 \text{Var}(\widehat{Y})$ and $\log_2 \varepsilon$ with $\varphi = \varphi_1$.}
    \label{Var-epsilon-Phi1} 
\end{subfigure}
\begin{subfigure}{.5\textwidth}
    \centering
    \includegraphics[height = 5cm]{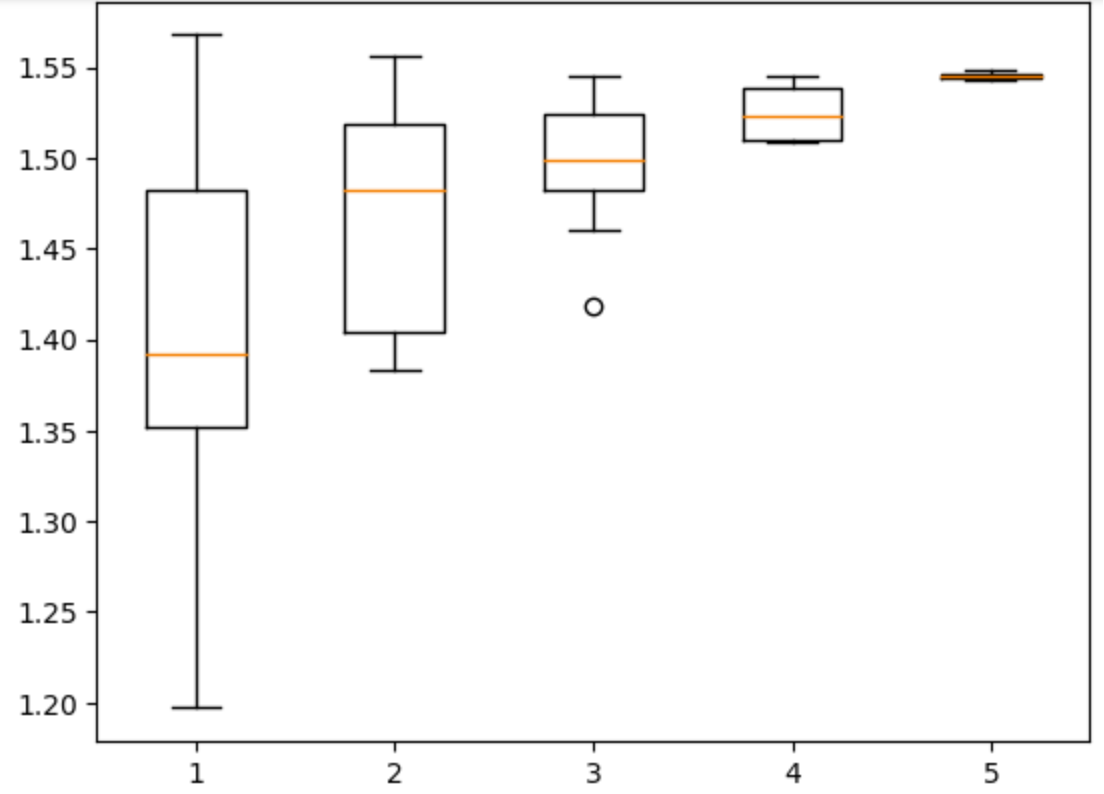}
    \caption{Relation between $\log_2 \text{Var}(\widehat{Y})$ and $\log_2 \varepsilon$ with $\varphi = \varphi_2$.}
    \label{Var-epsilon-Phi2}
\end{subfigure}
\begin{subfigure}{.5\textwidth}
    \centering
    \includegraphics[height = 5cm]{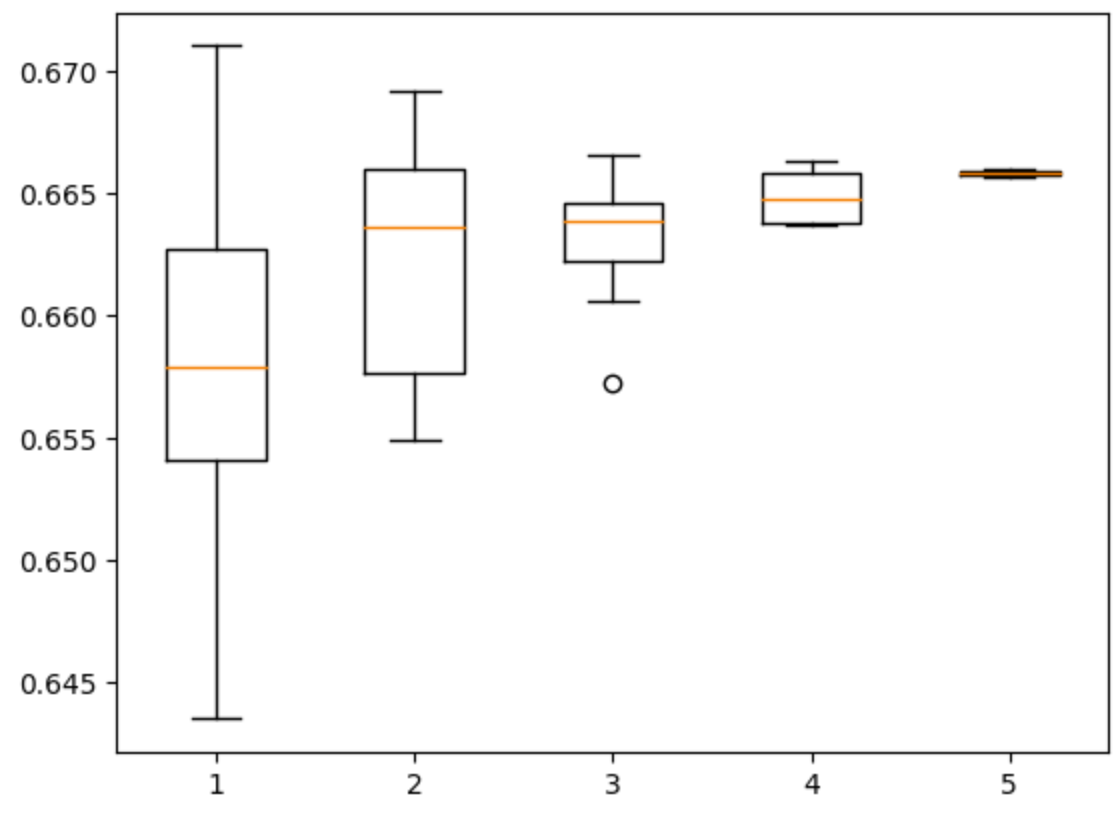}
    \caption{Relation between $\log_2 \text{Var}(\widehat{Y})$ and $\log_2 \varepsilon$ with $\varphi = \varphi_3$.}
    \label{Var-epsilon-Phi3}
\end{subfigure}
\begin{subfigure}{.5\textwidth}
    \centering
    \includegraphics[height = 5cm]{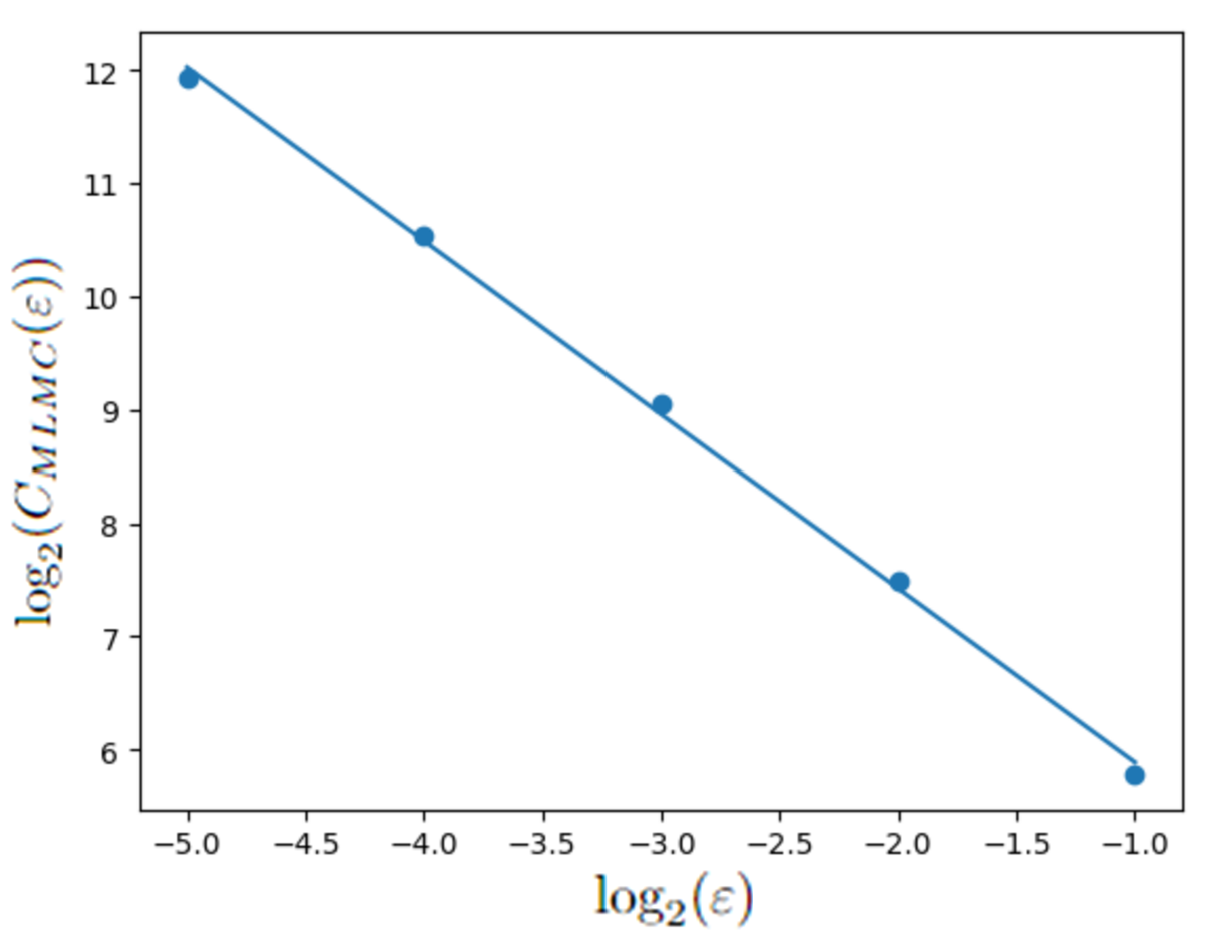}
\caption{Relation between $\log_2 C_{MLMC}(\varepsilon)$ and $\log_2 \varepsilon$.}
\label{ComputationalCost}
\end{subfigure}
\caption{Numerical results for invariant measure approximation.}
\label{SimInvaMeasAppro}
\end{figure}

\begin{table} 
\begin{center}
\begin{tabular}{c|ccc}
& Slope & y-intercept &	Coefficient of determination\\
\hline 
$\varphi = \varphi_1$	& 2.45	& -4.34 & 0.90\\
\hline 
$\varphi = \varphi_2$	& 2.39	& -3.48	& 0.90\\
\hline
$\varphi = \varphi_3$	& 2.59	& -10.53 & 0.94
\end{tabular}
\caption{Linear regression model for $\log_2 \text{Var}(\widehat{Y}) $ and $\log_2 \varepsilon$.}
\label{RegressionVariance}
\end{center}
\end{table}

In Figures \ref{Var-epsilon-Phi1}, \ref{Var-epsilon-Phi2}, \ref{Var-epsilon-Phi3}, and Table \ref{RegressionVariance}, we consider the relation between the total variance of the estimator $\widehat{Y}$ and the accuracy of the MLMC estimation $\varepsilon$. In the three figures, respectively, we draw the box-plots that represent the dispersion of the values of $\widehat{Y}$ for three different cases of the function $\varphi$, which are $\varphi = \varphi_1$, $\varphi = \varphi_2$, and $\varphi = \varphi_3$, all of which go from $\mathbb{R}^3$ to $\mathbb{R}$
\begin{align*}
    &\varphi_1(x_1, x_2, x_3) = x_1 + x_2 + x_3,\quad \varphi_2(x_1, x_2, x_3) = x_1^2 + x_2^2 + x_3^2, \quad \varphi_3(x_1, x_2, x_3) = \dfrac{|x_1 + x_2 + x_3|}{|x_1 + x_2 + x_3| + 1}.
\end{align*}
Assume that for a length of estimation interval $T_0$, number of levels $L_0$, and $(N^0_l)_{1 \leq l \leq L}$ samples for each level $l$, the approximation algorithm can attain an accuracy level $\varepsilon_0$. From this, we can choose the corresponding estimation interval, number of levels, and number of samples for each level $l$ so that the accuracy level can be attained at $\varepsilon_0$, $\varepsilon_0/2$, $\varepsilon_0/4$, $\varepsilon_0/8$, and $\varepsilon_0/16$. For these accuracy levels, we generate respectively $32$, $16$, $8$, $4$, and $2$ samples for $\widehat{Y}$ and compute the sample variances. At this point, we once again use the regression method to see the relation between $\log_2 \text{Var}(\widehat{Y}) $ and $\log_2 \varepsilon$. The slopes of the regression lines are shown in Table (\ref{RegressionVariance}), which are $2.45$, $2.39$, and $2.59$ for the three cases $\varphi = \varphi_1$, $\varphi = \varphi_2$, and $\varphi = \varphi_3$, respectively. This supports our theoretical result in Theorem \ref{MSEbound}.

Moreover, in Figure (\ref{ComputationalCost}), we plot the regression line to show how much computational cost it will take to attain a specific level of accuracy. To be more precise, for each of the four accuracy levels $\varepsilon$, $\varepsilon_0/2$, $\varepsilon_0/4$, and $\varepsilon_0/8$, we record the number of seconds it takes from the start till the end of the computation process. As shown in Theorem \ref{TheoryComputationalCost}, to achieve the accuracy of $\varepsilon$, the computational cost needed is $O(\varepsilon^{-2}|\log \varepsilon|^3)$. Indeed, the linear regression model for $\log_2 C_{MLMC}$ and $\log_2 \varepsilon$ is $y = -1.53x + 4.35$, with the coefficient of determination is $0.9985$.

\vskip 0.2cm 
\noindent 
\textbf{Acknowledgment} 

T.T. Kieu and D.T. Luong were supported by Hanoi National University of Education under Grant No. SPHN22-17. H.L. Ngo and N.K. Tran acknowledge support from the Vietnam Institute for Advanced Study in Mathematics (VIASM).

\end{document}